\documentclass[11pt,reqno]{amsart}
\usepackage{amsmath,amssymb,amsthm,amsfonts,mathrsfs,latexsym,times,color,bm}
\usepackage{amsmath,amsfonts}
\usepackage{amsthm}
\usepackage{graphicx}
\usepackage{color}
\usepackage{multicol}
\usepackage[notref,notcite]{showkeys}

\newtheorem{thm}{Theorem}[section]

\newtheorem{notation}{Notation}[section]
\newtheorem{Lemma}[thm]{Lemma}
\newtheorem{remark}{Remark}[section]
\newtheorem{theorem}[thm]{Theorem}

\newcommand{\x}{{\bf x}}

\numberwithin{equation}{section}

\newcommand{\beq}{\begin{equation}}
\newcommand{\eeq}{\end{equation}}
\newcommand{\ben}{\begin{eqnarray}}
\newcommand{\een}{\end{eqnarray}}
\newcommand{\beno}{\begin{eqnarray*}}
\newcommand{\eeno}{\end{eqnarray*}}

\newcommand{\bu}{\mathbf{u}}
\newcommand{\bb}{\mathbf{b}}
\newcommand{\bv}{\mathbf{v}}
\newcommand{\Bf}{\mathbf{f}}
\newcommand{\Bh}{\mathbf{h}}

\newcommand{\bF}{\mathbf{F}}
\newcommand{\bG}{\mathbf{G}}
\newcommand{\bH}{\mathbf{H}}
\newcommand{\bX}{\mathbf{X}}
\newcommand{\bW}{\mathbf{W}}
\newcommand{\bU}{\mathbf{U}}
\newcommand{\bB}{\mathbf{B}}

\numberwithin{equation}{section}

\begin{document}
\title[3D Hall-MHD equations with fractional Laplacians]{On 3D Hall-MHD equations with fractional Laplacians: global well-posedness}

\subjclass[2010]{35A01; 76W05}

\author{Huali  Zhang }
\address{Changsha University
of Science and Technology, Changsha, 410114, People's Republic of China.}
\email{zhlmath@yahoo.com}
\author{Kun Zhao }
\address{Department of Mathematics, Tulane University, New Orleans, LA 70118, United States.}
\email{kzhao@tulane.edu}

\keywords{Hall-MHD equations; fractional Laplacian; Cauchy problem; global well-posedness}

\begin{abstract}
The Cauchy problem for 3D incompressible Hall-magnetohydrodynamics (Hall-MHD) system with fractional Laplacians $(-\Delta)^{\frac{1}{2}}$ is studied. The well-posedness of 3D incompressible Hall-MHD equations remains an open problem with fractional diffusion $(-\Delta)^{\frac{1}{2}}$. First, global well-posedness of small-energy solutions with general initial data in $H^s$, $s>\frac{5}{2}$, is proved. Second, a special class of large-energy initial data is constructed, with which the Cauchy problem is globally well-posed. The proofs rely upon a new global bound of energy estimates involving Littlewood-Paley decomposition and Sobolev inequalities, which enables one to overcome the $\frac{1}{2}$-order derivative loss of the magnetic field.
\end{abstract}
	
\maketitle

\section{Introduction}

\subsection{Overview} This paper is oriented towards the Cauchy problem for 3D incompressible Hall-MHD equations with fractional Laplacians:	
\begin{equation}\label{HMHD}
\begin{cases}
\bu_t+(-\Delta)^{\alpha} \bu+\bu \cdot \nabla \bu+\nabla P - \bb \cdot \nabla \bb=\mathbf{0},\\
\bb_t+(-\Delta)^{\beta} \bb+\bu \cdot \nabla \bb-\bb \cdot \nabla \bu+ \nabla \times \left( \left(\nabla \times \bb \right) \times \bb\right)=\mathbf{0},\\
\nabla\cdot \bu=0, \quad \nabla\cdot \bb=0;\\
\bu|_{t=0}=\bu_0, \ \ \bb|_{t=0}=\bb_0,
\end{cases}
\end{equation}
$\mathbf{x}\in\mathbb{R}^3, t>0$. The object is to identify initial conditions under which \eqref{HMHD} is globally well-posed.

\subsection{Background} The classical 3D magnetohydrodynamics (MHD) equations with magnetic diffusion:
\begin{equation}\label{MHD}
\begin{cases}
\bu_t+ (-\Delta) \bu+\bu \cdot \nabla \bu+\nabla P - \bb \cdot \nabla \bb=\mathbf{0},\\
\bb_t+ (-\Delta) \bb+\bu \cdot \nabla \bb-\bb \cdot \nabla \bu=\mathbf{0},\\
\nabla\cdot \bu=0, \quad \nabla\cdot \bb=0,
\end{cases}
\end{equation}
$\mathbf{x}\in\mathbb{R}^3, t>0$, describe the macroscopic behavior of electrically conducting incompressible fluids in a magnetic field. Here, the unknown functions $\bu=\bu(\mathbf{x},t)\in\mathbb{R}^3$, $P=P(\mathbf{x},t)\in\mathbb{R}$, and $\bb=\bb(\mathbf{x},t)\in\mathbb{R}^3$ denote, respectively, the fluid velocity, pressure, and magnetic field.

In magnetic fusion theory, if a so-called ``Hall current'' perpendicular to the magnetic field is present in the equation for the current and if the magnetic field is very large, then the gyrofrequency of electrons greatly exceeds the average frequency of electron collision. This reflects that the Hall current effect has a greater influence on MHD fluids than the electrical resistivity does. However, in ideal MHD theory, the Hall-current term is routinely neglected in order to simplify the underlying mathematical analysis. In the 1960 paper entitled ``{\it Studies on Magnetohydrodynamic Waves and Other Anisotropic Wave Motions}'' \cite{L}, M.J. Lighthill criticized the neglect of the Hall term in ideal MHD theory, whose opinion implies that such a negligence breaks down the ideal MHD theory near the boundary region of any high-density plasma. As a remedy, Lighthill brought the Hall term back to the classical MHD equations, resulting in the standard Hall-MHD equations:
\begin{equation}\label{SHMHD}
\begin{cases}
\bu_t+ (-\Delta) \bu+\bu \cdot \nabla \bu+\nabla P - \bb \cdot \nabla \bb=\mathbf{0},\\
\bb_t+ (-\Delta) \bb+\bu \cdot \nabla \bb-\bb \cdot \nabla \bu+ \nabla \times \left( \left(\nabla \times \bb \right) \times \bb\right)=\mathbf{0},\\
\nabla\cdot \bu=0, \quad \nabla\cdot \bb=0.
\end{cases}
\end{equation}
Comparing with the classical MHD equations, \eqref{MHD}, the Hall term in \eqref{SHMHD} brings about more difficulties to the underlying analysis of qualitative behaviors of the standard Hall-MHD equations. Of interests to  present paper, Chae-Degond-Liu \cite{CDL} established global well-posedness of classical solutions for \eqref{SHMHD} with small initial data; Chae-Lee \cite{CL} improved the result of \cite{CDL} under weaker smallness assumptions on initial data; Fan-Huang-Nakamura \cite{FHN} studied global well-posedness of
axisymmetric solutions with large initial data; and Li-Yu-Zhu \cite{Li} and Zhang \cite{Zhang} constructed a class of (non-axisymmetric) large initial data and proved global well-posedness of classical solutions. There are also works dealing with low regularity solutions, extensibility criteria, and large-time asymptotic behavior for \eqref{SHMHD}, and we refer the readers to \cite{AD,CW,CW2,Duan,JO,KL,WZ,YM} for more information in these directions.

Despite their long-standing reputations in mathematical fluid mechanics, laboratory experiments and numerical simulations have suggested that classical models, such as the Navier-Stokes equations, can be modified to achieve better performance in real world applications. In particular, recent developments in mathematical fluid mechanics indicate that researchers have found it favorable to replace the Laplace operator $-\Delta$ by fractional powers of $-\Delta$. Such a generalization of ``normal" diffusion has been enforced to many classical fluid dynamics systems, including the Navier-Stokes equations \cite{LZ,W1}, Boussinesq equations \cite{HKR,HKR1,HZ}, MHD equations \cite{CW-old,W}, and surface quasi-geostrophic (SQG) equation \cite{CV,CCD,CCW,CV1,DL,KN,KNV}. Mathematical studies of these generalized models inspired invention of many new analytical methods, such as recent breakthrough in the studies of 2D SQG equation \cite{CV,CV1,KN,KNV}.

Inspired by the generalization of classical fluid dynamics models, fractional Laplacian has been implemented to \eqref{SHMHD}, giving rise to \eqref{HMHD}. In system \eqref{HMHD}, the exponents $\alpha$ and $\beta$ generally satisfy $0\le \alpha \le 1$ and $\beta\ge0$, and the fractional Laplacian $(-\Delta)^{\gamma}$ is defined through Fourier transform, namely, $\mathcal{F}((-\Delta)^{\gamma}f)(\bm{\xi})=|\bm{\xi}|^{2\gamma}\mathcal{F}(f)(\bm\xi)$. Rigorous mathematical study of \eqref{HMHD} was initiated by Chae-Wan-Wu \cite{CWW}, who established local well-posedness of classical solutions when there is no viscosity and $\beta \in (\frac{1}{2}, 1]$. The result of \cite{CWW} was improved by Dai \cite{Dai} through lowering the regularity of initial data. On the other hand, global well-posedness of classical solutions with small initial data is established when $\alpha,\beta \in [1,\frac{7}{6})$ by Pan-Ma-Zhu \cite{PMZ} and when $\alpha,\beta\in (1, \frac{3}{2})$ by Wu-Yu-Tang \cite{WYT}. There are also works concerning extensibility criteria and mild solutions, and we refer the readers to \cite{JZ,PZ, WZ1,WL,Ye} for detailed discussions on these topics.

\subsection{Motivation}
We observe that all of the previous studies on \eqref{HMHD} are concerned with the case when $\beta > \frac{1}{2}$. Since smaller $\beta$ generates weaker magnetic diffusion, analysis of \eqref{HMHD} when $\beta\in[0,\frac12]$ is more challenging than the case when $\beta>\frac12$. As a matter of fact, even in the case when $\beta=\frac12$,  well-posedness (local or global) of classical solutions to \eqref{HMHD} still remains open. Main difficulty stems from the $\frac{1}{2}$-order regularity loss of the magnetic field. To explain the difficulty, let us consider the simple situation when $\bu\equiv\mathbf{0}$. In this case, the equation for the magnetic field reads
$$
\bb_t+ (-\Delta)^{\frac12} \bb+ \nabla \times \left( \left(\nabla \times \bb \right) \times \bb\right)=\mathbf{0}.
$$
By standard commutator estimates, we can show that
\begin{equation*}
\begin{split}
&\frac{\mathrm{d}}{\mathrm{d}t}\|\bb\|^2_{H^s}+ \|\Lambda^{\frac{1}{2}}\bb\|^2_{H^s}\\
\lesssim \ &  \big| \int_{\mathbb{R}^3} \Lambda^s (( \nabla \times \bb) \times \bb)  \cdot \Lambda^s (\nabla \times \bb) \mathrm{d}\x  \big| +\text{lower order terms,}\\
\lesssim \ & \big|  \int_{\mathbb{R}^3} \big( \Lambda^s(( \nabla \times \bb) \times \bb) - \Lambda^s  (\nabla  \times \bb) \times \bb \big)\cdot \Lambda^s (\nabla \times \bb)  \mathrm{d}\x \big| +\text{lower order terms,}\\
\lesssim \ &  \|\bb\|_{H^s}\|\nabla \bb\|_{L^\infty} \|\nabla \bb\|_{H^s}+\text{lower order terms},
\end{split}
\end{equation*}
where $\Lambda=(-\Delta)^{\frac{1}{2}}$ and $s>0$. From above calculations we see that a global energy bound of the solution can not be achieved if one remains in the $H^s$-framework. However, if one could manage to show that
\begin{equation*}
\begin{split}
&\frac{\mathrm{d}}{\mathrm{d}t}\|\bb\|^2_{H^s}+ \|\Lambda^{\frac{1}{2}}\bb\|^2_{H^s} \lesssim  \|\nabla \bb\|_{L^\infty} \|\Lambda^{\frac{1}{2}} \bb\|^2_{H^s}+\text{lower order terms},
  \end{split}
\end{equation*}
then a global energy bound can be achieved at least in the regime of small-energy solutions.  Motivated by such an observation, we set our goal of this paper to identify initial conditions under which classical solutions to \eqref{HMHD} with $\beta=\frac{1}{2}$ are globally well-posed.

\subsection{Statement of Results}
Before stating our results, we explain some notations.

\begin{notation} Unless otherwise specified, $C$ denotes a generic constant which is independent of the unknown functions and time. For two positive quantities $f$ and $g$, $f\lesssim g$ means there exists $C>0$ such that $f\le Cg$, while $f\simeq g$ means there exists $C>0$ such that $C^{-1}g\le f\le Cg$. The symbol $[A, B]$ denotes the commutator: $[A,B]=AB-BA$. The symbol ${\Delta}_j$ denotes the homogeneous frequency localized operator with frequency $2^j, j \in \mathbb{Z}$. For $f \in H^s(\mathbb{R}^3)$, we denote $\|f\|_{H^s}:= \|f\|_{L^2}+\|f\|_{\dot{H}^s}$, where $\|f\|^2_{\dot{H}^s} := {\sum_{j \geq -1}} 2^{2js}\|{\Delta}_j f\|^2_{L^2} $. We also denote $\|f\|^r_{\dot{B}^s_{p,r}} := {\sum_{j \geq -1}}2^{jsr}\|{\Delta}_j f\|^r_{L^p}$.
\end{notation}

Global well-posedness of classical solutions to \eqref{HMHD} for general initial data with small energy is established in the following theorem.

\begin{theorem}\label{thm1}
Consider the Cauchy problem \eqref{HMHD} with $0 \leq \alpha \leq 1$ and $\beta=\frac{1}{2}$. Let $s>\frac52$. Suppose that the initial data satisfy $\nabla \cdot \bu_0=\nabla \cdot \bb_0=0$, and
\begin{equation}\label{i1}
  \|\bu_0\|_{H^s}+\|\bb_0\|_{H^s} \leq \epsilon,
\end{equation}
for some constant $\epsilon>0$. If $\epsilon$ is sufficiently, then there exists a unique and global-in-time solution $(\bu,\bb)$ to \eqref{HMHD}. Moreover, the solution satisfies the following energy estimate:
\begin{equation}\label{i2}
\begin{split}
&\|\bu(t) \|^2_{H^s}+ \| \bb(t) \|^2_{H^s}+ \int^t_0 \| \Lambda^{\alpha} \bu(\tau) \|^2_{H^s} \mathrm{d} \tau+ \int^t_0 \| \Lambda^{\frac{1}{2}} \bb(\tau) \|^2_{H^s} \mathrm{d} \tau
 \lesssim \epsilon.
  \end{split}
\end{equation}
\end{theorem}
\begin{remark}
The proof of Theorem \ref{thm1} crucially relies on a global bound of energy estimates. After getting the global bound, we then use iteration method to get the local existence of small solutions. As we referred before, it may occur $\frac12$-order derivative loss if we use the standard energy energy estimates and classical commutator estimates. Therefore, we need to use concrete Littlewoold-Palay decomposition to transfor the derivatives to avoid any derivatives loss. It's a non-trival process, please refer Section 3. But, if $\beta=\frac12$, it remains an open question for there is a unique local solution for general initial data.
\end{remark}

The second result shows that classical solutions to \eqref{HMHD} are global well-posed for a special class of initial data carrying potentially large $H^s$ energy. Our work is motivated by similar discussions on the Navier-Stokes equations \cite{LLZ}, SQG equation \cite{LPW, Zhang2}, classical MHD equations \cite{DTW,LZZ,ZZ0} and standard Hall-MHD equations \cite{Li,Zhang}.

\begin{theorem}\label{thm2}
Consider the Cauchy problem \eqref{HMHD} with $0 \leq \alpha \leq 1$ and $\beta=\frac{1}{2}$. Let $s>\frac52$.  Let $\bv_0$ be a three-dimensional vectorfield satisfying the following properties:
\begin{align*}\label{103}
&\nabla \cdot \bv_0=0, \quad \nabla \times \bv_0=\sqrt{-\Delta}\,\bv_0,\\
& \mathrm{supp}\, \widehat{\bv}_0 \subseteq \{1-\varepsilon \leq |\bm\xi| \leq 1+\varepsilon \}, \quad 0<\varepsilon \leq \frac{1}{2},
\end{align*}
where $\widehat{\bv}_0$ denotes the Fourier transform of $\bv_0$ and $\varepsilon$ is a small parameter. Suppose that the initial functions can be decomposed as  	
$\bu_0=\bu_{01}+\bu_{02},\ \bb_0=\bb_{01}+\bb_{02}$, where the functions in the decomposition satisfy the following conditions:
\begin{equation*}\label{100}
\begin{aligned}
&\nabla\cdot\bu_{01}=0, &\quad&\nabla\cdot\bb_{01}=0,\\
&\bu_{02}=\alpha_1 \bv_0, &\quad& \bb_{02}=\alpha_2 \bv_0,
\end{aligned}
\end{equation*}
for some constants $\alpha_1$ and $\alpha_2$. Suppose that $\bu_{01}$, $\bb_{01}$ and $\bv_0$ satisfy
\begin{equation}\label{111}
\left( \|\bu_{01}\|^2_{H^{s}}+\|\bb_{01}\|^2_{H^{s}}+ \big(\varepsilon^{s+\frac{1}{2}} +\varepsilon\big) \|\widehat{\bv}_0\|_{L^2_{\xi}} \|\widehat{\bv}_0\|_{L^1_{\bm\xi}} +\|\widehat{\bv}_0\|_{L^1_{\bm\xi}} \right) \exp \big( C \|\widehat{\bv}_0\|_{L^1_{\bm\xi}} \big) \leq \delta
\end{equation}	
for some constant $\delta>0$ and a generic constant $C>0$. Then there exists a unique and global-in-time solution to \eqref{HMHD}, provided $\varepsilon$ and $\delta$ are sufficiently small.
\end{theorem}

\begin{remark}
To introduce initial data with large $H^s$ energy, we follow the idea of \cite{LLZ}. Let $\mathbf{n}(\bm\xi)$ be a smooth vectorfield satisfying $\bm\xi \cdot \mathbf{n}(\bm\xi)=0$ and $|\mathbf{n}(\bm\xi)|=1$. For $0< \varepsilon <1$, let $\psi$ be a smooth function in $\mathcal{S}(\mathbb{R}^3)$ with $\mathrm{supp}\, \psi \ \subseteq \{1-\varepsilon < |\bm\xi|< 1+\varepsilon \}$, where $\mathcal{S}(\mathbb{R}^3)$ denotes the Schwartz space. Let
\begin{equation*}
  \mathbf{g}(\x)=\int_{1-\varepsilon<|\bm\xi|<1+\varepsilon} \big(\mathbf{n}(\bm\xi)\sin(\x\cdot \bm\xi)+|\bm\xi|^{-1}\bm\xi \times \mathbf{n}(\bm\xi)\cos(\x\cdot \bm\xi) \big) \psi(\bm\xi)\mathrm{d}\bm\xi.
\end{equation*}
Then it can be verified (c.f.\,\cite{LLZ}) that
$$
\nabla \cdot \mathbf{g}=0, \quad  \nabla \times \mathbf{g}= \sqrt{-\Delta}\,\mathbf{g}.
$$
Define
\begin{equation*}
   \bv_0=\varepsilon^{-\frac{3}{2}}\log (\varepsilon^{-1}) \mathbf{g}.
\end{equation*}
By direct calculations, we can show that
$$
\begin{aligned}
&\qquad\widehat{\bv}_0=\varepsilon^{-\frac{3}{2}}\log (\varepsilon^{-1}) |\bm\xi|^{-1} (\bm\xi \times \mathbf{n}(\bm\xi))\psi(\bm\xi),\\
&\|{\widehat \bv}_0\|_{L^2_{\bm\xi}} \simeq \varepsilon^{-\frac{1}{2}}\log (\varepsilon^{-1}),
 \quad \|{\widehat \bv}_0\|_{L^1_{\bm\xi}} \simeq \varepsilon^{\frac{1}{2}} \log (\varepsilon^{-1}).
 \end{aligned}
$$
Hence, the quantity involving $\widehat{\bv}_0$ on the left-hand side of \eqref{111} is quantitatively equivalent to
\begin{equation*}
\big[\big(\varepsilon^{s+\frac{1}{2}} +\varepsilon\big) (\log(\varepsilon^{-1}))^2 + \varepsilon^{\frac{1}{2}} \log (\varepsilon^{-1})\big] \big(1+\exp \big\{ \varepsilon^{\frac{1}{2}} \log (\varepsilon^{-1}) \big\}\big).
\end{equation*}
Thus, \eqref{111} can be realized by choosing $\varepsilon$ and $\|(\bu_{01},\bb_{01})\|^2_{H^{s}}$ to be sufficiently small and taking for example $\delta \simeq \varepsilon^{\frac13}$. In this case, $\|(\bu_0,\bb_0)\|_{H^s}$ can be arbitrarily large, since $\|{\widehat \bv}_0\|_{L^2_{\bm\xi}} \simeq \varepsilon^{-\frac{1}{2}}\log (\varepsilon^{-1})$.
\end{remark}

\begin{remark}
In the limiting case when $\varepsilon=0$, it was pointed out in \cite{LLZ} that the function $\mathbf{g}$ satisfies $\nabla\cdot\mathbf{g}=0$ and $\nabla\times \mathbf{g}=\mathbf{g}$, i.e., $\mathbf{g}$ is a Beltrami flow and so is $\bv_0$. In this case, $\bu_{02}=\alpha_1\bv_0$ is a Beltrami flow and $\bb_{02}=\alpha_2\bv_0$ is called a force-free field. Thus, our initial data can be viewed as small perturbations near a potentially large Beltrami flow and a potentially large force-free field.
\end{remark}

We prove Theorem \ref{thm1} by utilizing standard $L^p$-based energy methods and Littlewood-Paley decomposition. A special ingredient is transferring derivatives in nonlinear terms, which is inspired by \cite{CWW}. The proof of Theorem \ref{thm2} relies on analyzing the nonlinear structure of the helicity and carefully crafting the nonlinear smallness of initial data.

The rest of the paper is organized as follows. In Section 2, we introduce and prove some technical lemmas which are utilized in the proofs of the main results. We then prove Theorem \ref{thm1} and Theorem \ref{thm2} in Section 3 and Section 4, respectively.

\section{Preliminaries}

We first introduce some commutator estimates.
\begin{Lemma}\label{san}
Let $\bH,\bX,\bW$ be divergence free vectorfields in $\mathbb{R}^3$. For $j\ge-1$, let
\begin{equation*}
I=\int_{\mathbb{R}^3} [{\Delta}_j, \bH \cdot \nabla]\bX \cdot {\Delta}_j \bW \mathrm{d}\x.
\end{equation*}
Then it holds that
\begin{align}\label{q}
  |I| \lesssim \ &\|\nabla \bH\|_{L^\infty} \|{\Delta}_j \bX\|_{L^2}\|{\Delta}_j \bW\|_{L^2}+\|{\Delta}_j \bH\|_{L^\infty} \|\nabla \bX\|_{L^2}\|{\Delta}_j \bW\|_{L^2} \nonumber\\
  & +\|\nabla \bH\|_{L^\infty}\|{\Delta}_j \bW\|_{L^2} {\sum_{k \geq j-1}} 2^{j-k} \|{{\Delta}}_k \bX\|_{L^2}.
\end{align}
\end{Lemma}

\begin{proof}
By phase decomposition, we first write $I$ as $I=I_{1}+I_{2}+I_{3}$, where
\begin{equation*}
  \begin{split}
  &I_{1}=\sum_{|k-j|\leq 2}\int_{\mathbb{R}^3} \left( {\Delta}_j(S_{k-1}\bH \cdot \nabla {\Delta}_k \bX )  - S_{k-1}\bH \cdot \nabla {\Delta}_j{\Delta}_k \bX \right) \cdot {\Delta}_j \bW \mathrm{d}\x,
  \\
  &I_{2}=\sum_{|k-j|\leq 2}\int_{\mathbb{R}^3} \left( {\Delta}_j({\Delta}_k \bH \cdot \nabla S_{k-1} \bX )  - {\Delta}_k \bH \cdot \nabla {\Delta}_jS_{k-1}\bX \right) \cdot {\Delta}_j \bW \mathrm{d}\x,
  \\
  &I_{3}=\sum_{k\geq j-1}\int_{\mathbb{R}^3} \big( {\Delta}_j({\Delta}_k \bH \cdot \nabla{\tilde{\Delta}}_k \bX )  - {\Delta}_k \bH \cdot \nabla \Delta_j{\tilde{\Delta}}_k \bX \big) \cdot {\Delta}_j \bW \mathrm{d}\x,
  \end{split}
\end{equation*}
where $S_k:={\sum}_{k' \leq k-1}{\Delta}_{k'}$ and ${\tilde{\Delta}}_k:={\Delta}_{k-1}+{\Delta}_{k}+{\Delta}_{k+1}$. Using H\"older's inequality and commutator estimate, we can show that
\begin{align}\label{I1}
  |I_{1}|
   &\lesssim {\sum_{|k-j|\leq 2}} \|\nabla S_{k-1}\bH\|_{L^\infty}\| \Delta_k \bX\|_{L^2}\| \Delta_j \bW\|_{L^2}\nonumber
   \\
   &\lesssim \|\nabla \bH\|_{L^\infty}\| \Delta_j \bX\|_{L^2}\| \Delta_j \bW\|_{L^2}.
\end{align}
Similarly, we can show that
\begin{align}\label{I2}
  |I_{2}| &\lesssim {\sum_{|k-j|\leq 2}} \|\nabla S_{k-1}\bX\|_{L^2}\| \Delta_k \bH\|_{L^\infty}\| \Delta_j \bW\|_{L^2}\nonumber
  \\
  &\lesssim \|\nabla \bX\|_{L^2}\| \Delta_j \bH\|_{L^\infty}\| \Delta_j \bW\|_{L^2}.
\end{align}
Since
\begin{equation*}
  I_{3}=\sum_{k \geq j-1} \int_{\mathbb{R}^3} [\Delta_j, \Delta_k \bH \cdot \nabla]\tilde{\Delta}_k \bX \cdot \Delta_j \bW \mathrm{d}\x,
\end{equation*}
it holds that
\begin{equation}\label{4}
  |I_{3}| \leq {\sum_{k \geq j-1}} \big\| [\Delta_j, \Delta_k \bH \cdot \nabla]\tilde{\Delta}_k \bX \big\|_{L^2} \|\Delta_j \bW \|_{L^2}.
\end{equation}
Using $\nabla \cdot \bH=0$ and Bernstein's inequality, we deduce
\begin{align}\label{3}
  {\sum_{k \geq j-1}}\big\| [\Delta_j, \Delta_k \bH \cdot \nabla]\tilde{\Delta}_k \bX \big\|_{L^2} &\leq {\sum_{k \geq j-1}} \big\| \nabla \cdot [\Delta_j,\Delta_k \bH]\tilde{\Delta}_k \bX\big\|_{L^2}\nonumber
  \\
  & \lesssim {\sum_{k \geq j-1}} 2^j \|\Delta_k \bH\|_{L^\infty} \|\tilde{\Delta}_k \bX\|_{L^2}\nonumber
  \\
   & \lesssim {\sum_{k \geq j-1}} 2^{k-j} \|\nabla \bH\|_{L^\infty} \|{\Delta}_k \bX\|_{L^2}.
\end{align}
Substituting \eqref{3} into \eqref{4}, we have
\begin{equation}\label{I3}
  |I_{3}| \lesssim \|\nabla \bH\|_{L^\infty} \|\Delta_j \bW\|_{L^2} {\sum_{k \geq j-1}} 2^{k-j}\|{\Delta}_k \bX\|_{L^2}.
\end{equation}
Combining \eqref{I1}, \eqref{I2}, and \eqref{I3}, we arrive at \eqref{q}.
\end{proof}

The second lemma is utilized in the proof of Theorem \ref{thm2}, which provides estimate of the linearized flow associated with \eqref{HMHD}. To state the lemma, we let $\bU$ be the solution of
\begin{equation}\label{U}
\begin{cases}
\bU_t+ (-\Delta)^\alpha \bU=\mathbf{0},\\
\bU(\x,0)=\alpha_1 \bv_0,
\end{cases}
\end{equation}
and let $\bB$ be the solution of
\begin{equation}\label{B}
\begin{cases}
\bB_t+ (-\Delta)^{\frac{1}{2}} \bB=\mathbf{0},\\
\bB(\x,0)=\alpha_2 \bv_0,
\end{cases}
\end{equation}
where $\bv_0$ is the same function as specified in Theorem \ref{thm2}. Then we have

\begin{Lemma}\label{L1}
Let $\bU$ and $\bB$ be solutions to \eqref{U} and \eqref{B}, respectively. Then the following estimates
\begin{equation}\label{U1}
\begin{split}
 \| {\bU}\times (\nabla \times {\bU})\|_{H^{s+\frac{1}{2}}}+\| {\bB}\times (\nabla \times {\bB})\|_{H^{s+\frac{1}{2}}} \lesssim \varepsilon^{s+\frac{1}{2}} ( e^{- \frac{ t}{2^{2\alpha}}}+ e^{-\frac{ t}{2}  } )^2\|\widehat{\bv}_0\|_{L^2_{\bm\xi}}\|\widehat{\bv}_0\|_{L^1_{\bm\xi}}
 \end{split}
\end{equation}
and
\begin{equation}\label{UB}
\begin{split}
&\int^\infty_0 \|\nabla \times (\bU \times \bB) \|_{H^{s+\frac{1}{2}}} (t)\mathrm{d}t  \lesssim \varepsilon\|\widehat{\bv}_0\|_{L^2_{\bm\xi}}\|\widehat{\bv}_0\|_{L^1_{\bm\xi}}
\end{split}
\end{equation}
hold.
\end{Lemma}

\begin{proof}
First of all, note that since
\begin{align*}
& \nabla \cdot \bv_0=0, \qquad  \ \nabla \times \bv_0=\sqrt{-\Delta}\bv_0,\\
& \bU=\alpha_1 e^{- t (-\Delta)^\alpha}\bv_{0}, \quad \bB=\alpha_2e^{- t (-\Delta)^{\frac{1}{2}}}\bv_{0},
\end{align*}
then $\bU$ and $\bB$ satisfy the following properties:
\begin{align*}
&\nabla \cdot \bU=0, \quad \quad \nabla \times \bU=\sqrt{-\Delta}\, \bU,\\
& \nabla \cdot \bB=0, \quad \quad \nabla \times \bB=\sqrt{-\Delta}\, \bB.
\end{align*}
To prove \eqref{U1}, we choose a $C^\infty(\mathbb{R}^3)$ cut-off function $\gamma(\bm\xi)$ such that $\gamma\equiv1$ on the support of $\bv_0$, and $\gamma(\bm\xi)\equiv0$ if $|\bm\xi|\geq 1+2\varepsilon$ or $|\bm\xi|\leq 1-2\varepsilon$. Then we have
\begin{align}\label{UB1}
  &\bU(\x,t)=\alpha_1{\mathcal{F}}^{-1}\left(e^{- |\bm\xi|^{2\alpha} t}\gamma(\bm\xi)\right) \ast \bv_0,\nonumber \\
  &\bB(\x,t)=\alpha_2{\mathcal{F}}^{-1}\left(e^{-|\bm\xi|t}\gamma(\bm\xi)\right) \ast \bv_0,
\end{align}
where $\mathcal{F}^{-1}(\cdot)$ denotes the inverse Fourier transform. Since $\bU\times \bU=\mathbf{0}$ and $\bB \times \bB=\mathbf{0}$, then we have
\begin{align*}
&\|\bU \times (\nabla \times \bU)\|_{H^{s+\frac{1}{2}}} + \|\bB \times (\nabla \times \bB)\|_{H^{s+\frac{1}{2}}}\\
= &\|\bU \times (\nabla \times \bU-\bU)\|_{H^{s+\frac{1}{2}}} + \|\bB \times (\nabla \times \bB-\bB)\|_{H^{s+\frac{1}{2}}}.
\end{align*}
By H\"older's inequality, we obtain
\begin{align}\label{205}
&\|\bU \times (\nabla \times \bU)\|_{H^{s+\frac{1}{2}}} + \|\bB \times (\nabla \times \bB)\|_{H^{s+\frac{1}{2}}} \nonumber\\
\leq  \ & \|\bU\|_{W^{s+\frac{1}{2},\infty}} \|\nabla \times \bU-\bU\|_{H^{s+\frac{1}{2}}}+\|\bB\|_{W^{s+\frac{1}{2},\infty}}\|\nabla \times \bB-\bB\|_{H^{s+\frac{1}{2}}}.
\end{align}
Using the information about the support of $\bU$ and $\bB$ (pertaining to the support of $\bv_0$), we can show that
\begin{equation*}
\begin{split}
\|\bU\|_{W^{s+\frac{1}{2},\infty}}+\|\bB\|_{W^{s+\frac{1}{2},\infty}} & \lesssim \gamma(\bm\xi) \cdot|\bm\xi|^{s+\frac{1}{2}}\cdot(\|\widehat{\bU}\|_{L^1_{\bm\xi}}+\|\widehat{\bB}\|_{L^1_{\bm\xi}})\\
&\lesssim \gamma(\bm\xi)\cdot( \alpha_1e^{- |\bm\xi|^{2\alpha} t}+ \alpha_2 e^{- |\bm\xi| t} ) \cdot |\bm\xi|^{s+\frac{1}{2}}\|\widehat{\bv}_0\|_{L^1_{\bm\xi}}\\
& \lesssim ( e^{- \frac{ t}{2^{2\alpha}}}+ e^{-\frac{ t}{2}  } ) \|\widehat{\bv}_0\|_{L^1_{\bm\xi}}.
\end{split}
\end{equation*}
Similarly, we can show that
\begin{equation*}
\begin{split}
\|\nabla \times \bU-\bU\|_{H^{s+\frac{1}{2}}}+\|\nabla \times \bB-\bB\|_{H^{s+\frac{1}{2}}} &\lesssim \gamma(\bm\xi)\cdot(|\bm\xi|-1)^{s+\frac{1}{2}} (\|\widehat{\bB}\|_{L^2_{\bm\xi}}+\|\widehat{\bU}\|_{L^2_{\bm\xi}})\\
&\lesssim \varepsilon^{s+\frac{1}{2}} ( e^{- \frac{ t}{2^{2\alpha}}}+ e^{-\frac{ t}{2} } )\|\widehat{\bv}_0\|_{L^2_{\bm\xi}}.
  \end{split}
\end{equation*}
Substituting the above inequalities into \eqref{205}, we obtain
\begin{equation}\label{206}
\begin{split}
\|\bU \times (\nabla \times \bU)\|_{H^{s+\frac{1}{2}}} + \|\bB \times (\nabla \times \bB)\|_{H^{s+\frac{1}{2}}}
\lesssim \varepsilon^{s+\frac{1}{2}} ( e^{- \frac{ t}{2^{2\alpha}}}+ e^{-\frac{ t}{2}  } )^2\|\widehat{\bv}_0\|_{L^2_{\bm\xi}}\|\widehat{\bv}_0\|_{L^1_{\bm\xi}}.
\end{split}
\end{equation}
Hence, \eqref{U1} is proved.

To show \eqref{UB}, we observe that since $\text{supp}\,\widehat{\bU \times \bB} \subseteq \text{supp}\, \widehat{\bU}+\text{supp}\, \widehat{\bB}$, it holds that
\begin{equation}\label{q}
  \text{supp}\, \widehat{\bU \times \bB} \subseteq \left\{ \bm\xi: |\bm\xi| \leq 2+2\varepsilon  \right\}, \ 0 < \varepsilon \leq \frac{1}{2}.
\end{equation}
Hence, we have
\begin{equation}\label{208}
\|\nabla \times (\bU \times \bB) \|_{H^{s+\frac{1}{2}}} \lesssim  \|\widehat{\bU \times \bB} \|_{L^2_{\bm\xi}}.
\end{equation}
By a direct calculation, we can show that
\begin{align}\label{f01}
\widehat{\bU \times \bB} & = \alpha_1 \alpha_2 \int_{\mathbb{R}^3} e^{-|\bm\xi-\bm\eta|^{2\alpha} t} \widehat{\bv}_0(\bm\xi-\bm\eta) \times e^{-|\bm\eta| t} \widehat{\bv}_0(\bm\eta)\mathrm{d}\bm\eta
\nonumber\\
&= \alpha_1 \alpha_2 \int_{\mathbb{R}^3} e^{-t(|\bm\xi-\bm\eta|^{2\alpha}+|\bm\eta|) } \widehat{\bv}_0(\bm\xi-\bm\eta) \times \widehat{\bv}_0(\bm\eta)\mathrm{d}\bm\eta.
\end{align}
On the other hand, using the convolution theorem, we can show that
\begin{align}\label{f00}
\widehat{\bU \times \bB} & = \alpha_1 \alpha_2 \int_{\mathbb{R}^3} e^{-|\tilde{\bm\eta}|^{2\alpha} t} \widehat{\bv}_0(\tilde{\bm\eta}) \times e^{-|\bm\xi-\tilde{\bm\eta}| t} \widehat{\bv}_0(\bm\xi-\tilde{\bm\eta})\mathrm{d}\tilde{\bm\eta}.
\nonumber\\
&= \alpha_1 \alpha_2 \int_{\mathbb{R}^3} e^{-t(|\tilde{\bm\eta}|^{2\alpha}+|\bm\xi-\tilde{\bm\eta}| ) } \widehat{\bv}_0(\tilde{\bm\eta}) \times \widehat{\bv}_0(\bm\xi-\tilde{\bm\eta})\mathrm{d}\tilde{\bm\eta}.
\end{align}
Since $\widehat{\bv}_0(\tilde{\bm\eta}) \times \widehat{\bv}_0(\bm\xi-\tilde{\bm{\eta}})=-\widehat{\bv}_0(\bm\xi-\tilde{\bm\eta}) \times \widehat{\bv}_0(\tilde{\bm\eta})$, we obtain from \eqref{f00} that
\begin{align}\label{f000}
\widehat{\bU \times \bB}
&=-\alpha_1 \alpha_2 \int_{\mathbb{R}^3} e^{-t( |\tilde{\bm{\eta}}|^{2\alpha} + |\bm\xi-\tilde{\bm{\eta}}|)} \widehat{\bv}_0(\bm \xi-\tilde{\bm{\eta}}) \times \widehat{\bv}_0(\tilde{\bm{\eta}})\mathrm{d}\tilde{\bm{\eta}} \nonumber\\
&= -\alpha_1 \alpha_2 \int_{\mathbb{R}^3} e^{-t( |{\bm{\eta}}|^{2\alpha} + |\bm\xi-\bm{\eta}|)} \widehat{\bv}_0(\bm \xi-\bm{\eta}) \times \widehat{\bv}_0(\bm{\eta})\mathrm{d}\bm{\eta}.
\end{align}
Adding \eqref{f01} and \eqref{f000}, we obtain
\begin{align}\label{f0}
\widehat{\bU \times \bB} & = \alpha_1 \alpha_2 \int_{\mathbb{R}^3} e^{-|\bm\xi-{\bm\eta} |^{2\alpha} t} \widehat{\bv}_0(\bm\xi-{\bm\eta}) \times e^{-|\tilde{\bm\eta}| t} \widehat{\bv}_0({\bm\eta})\mathrm{d}{\bm\eta}\nonumber \\
&= \frac{1}{2}\alpha_1 \alpha_2 \int_{\mathbb{R}^3} Q(t,\bm\xi,\bm\eta) \widehat{\bv}_0(\bm\xi-\bm\eta) \times  \widehat{\bv}_0(\bm\eta)\mathrm{d}\bm\eta,
\end{align}
where
$$
Q(t,\bm\xi,\bm\eta):= e^{-t(|\bm\xi-\bm\eta|^{2\alpha}+|\bm\eta|) }-e^{-t(|\bm\xi-\bm\eta|+|\bm\eta|^{2\alpha}) }.
$$
Note that within the support of  $\widehat{\bv}_0(\bm\xi-\bm\eta) \times  \widehat{\bv}_0(\bm\eta)$, it holds that $ |\bm\xi-\bm\eta|, |\bm\eta| \in [1-\varepsilon, 1+\varepsilon]$, where $0< \varepsilon < \frac{1}{2}$. Hence, using Taylor's theorem, we can show that
\begin{align}\label{209}
|Q(t,\bm\xi,\bm\eta)|\leq \ & e^{-|\bm\xi-\bm\eta|^{2\alpha} t} \big|e^{-|\bm\eta|t}-e^{-|\bm\xi-\bm\eta|t} \big|+e^{-|\bm\xi-\bm\eta| t} \big|e^{-|\bm\xi-\bm\eta|^{2\alpha} t}-e^{-|\bm\eta|^{2\alpha} t} \big|\nonumber\\
\lesssim \ & t\, e^{- |\bm\xi-\bm\eta|^{2\alpha} t} \big| |\bm\xi-\bm\eta|-|\bm\eta|\big|+C\, t\, e^{- |\bm\xi-\bm\eta| t} \big| |\bm\xi-\bm\eta|^{2\alpha} -|\bm\eta|^{2\alpha}\big|\nonumber\\
\lesssim \ & e^{-\frac{1}{2} |\bm\xi-\bm\eta|^{2\alpha} t} \frac{\big| |\bm\xi-\bm\eta|-|\bm\eta|\big|}{|\bm\xi-\bm\eta|^{2\alpha}}+C e^{-\frac{1}{2} |\bm\xi-\bm\eta| t} \frac{\big| |\bm\xi-\bm\eta|^{2\alpha} -|\bm\eta|^{2\alpha}\big|}{|\bm\xi-\bm\eta|}\nonumber\\
\lesssim \ & (e^{-\frac{1}{2}|\bm\xi-\bm\eta|^{2\alpha} t}  +  e^{-\frac{1}{2} |\bm\xi-\bm\eta| t})\,\varepsilon\nonumber\\
\lesssim \ & (e^{-\frac{1}{2}\cdot {(\frac{1}{2}})^{2\alpha} t}  +  e^{-\frac{1}{2} \cdot {\frac{1}{2}} t})\varepsilon.
\end{align}
Integrating \eqref{208} with respect to time and using \eqref{f0}, we derive that
\begin{equation*}
\begin{split}
\int^\infty_0 \|\nabla \times ( \bU \times \bB) \|_{H^{s+\frac{1}{2}}}\mathrm{d}t &\lesssim \int^\infty_0 \|\widehat{ \bU \times \bB }\|_{L^2_{\bm\xi}}\mathrm{d}t \\
& \lesssim  \frac{1}{2}\alpha_1 \alpha_2 \int^\infty_0 \Big\| \int_{\mathbb{R}^3} Q(t,\bm\xi,\bm\eta) \widehat{\bv}_0(\bm\xi-\bm\eta) \times  \widehat{\bv}_0(\bm\eta)d\bm\eta \Big\|_{L^2_{\bm\xi}}\mathrm{d}t.
\end{split}
\end{equation*}
By Young's inequality and \eqref{209}, we conclude that
\begin{equation*}
\begin{split}
\int^\infty_0 \|\nabla \times ( \bU \times \bB) \|_{H^{s+\frac{1}{2}}}\mathrm{d}t
&\lesssim  \int^\infty_0 \|Q(t, \cdot)\|_{L^\infty} \|\widehat{\bv}_0\|_{L^2_{\bm\xi}}\|\widehat{\bv}_0\|_{L^1_{\bm\xi}}\mathrm{d}t\\
&\lesssim \|\widehat{\bv}_0\|_{L^2_{\bm\xi}}\|\widehat{\bv}_0\|_{L^1_{\bm\xi}} \int^\infty_0\varepsilon \big( (e^{-\frac{1}{2}\cdot {(\frac{1}{2}})^{2\alpha} t}  +  e^{-\frac{1}{2} \cdot {\frac{1}{2}} t} \big) \mathrm{d}t\\
&\lesssim \varepsilon\|\widehat{\bv}_0\|_{L^2_{\bm\xi}}\|\widehat{\bv}_0\|_{L^1_{\bm\xi}}.
\end{split}
\end{equation*}
This completes the proof of Lemma \ref{L1}.
\end{proof}

\section{Proof of Theorem \ref{thm1}}

This section is devoted to the proof of Theorem \ref{thm1}.
\subsection{A prior energy estimates}
Firstly, we will give a prior energy estimates, which is separated into several steps.

{\bf Step 1.} Taking $L^2$ inner product of the equations in \eqref{HMHD} with $(\bu,\bb)$, we can show that
\begin{equation}\label{E0}
  \|\bu(t)\|^2_{L^2}+\|\bb(t)\|^2_{L^2}+ \int^t_0 ( \|\Lambda^{\alpha} \bu(\tau)\|^2_{L^2}+ \|\Lambda^{\frac{1}{2}} \bb(\tau)\|^2_{L^2})\mathrm{d}\tau \leq \|\bu_0\|^2_{L^2}+\|\bb_0\|^2_{L^2}.
\end{equation}

{\bf Step 2.} For $j\ge -1$, taking ${\Delta}_j$ to the equations in \eqref{HMHD}, we see that
\begin{equation*}
\begin{split}
& {\partial}_t{\Delta}_j \bu+ {\Delta}_j (-\Delta)^{\alpha} \bu+ {\Delta}_j(\bu \cdot \nabla \bu) + \nabla {\Delta} _j P={\Delta}_j(\bb\cdot \nabla \bb),\\
& {\partial}_t {\Delta}_j \bb+ {\Delta}_j(\bu \cdot \nabla \bb)+{\Delta}_j \left( \nabla \times \left( \left(\nabla \times \bb \right) \times \bb\right) \right)+  {\Delta}_j (-\Delta)^{\frac{1}{2}}\bb={\Delta}_j(\bb \cdot \nabla \bu).
\end{split}
\end{equation*}
Taking the $L^2$ inner product of the above equations with $({\Delta}_j \bu, {\Delta}_j \bb)$, we have
\begin{equation}\label{000}
\begin{split}
\frac{1}{2} \frac{\mathrm{d}}{\mathrm{d}t} \big( \|{\Delta}_j \bu\|^2_{L^2} + \|{\Delta}_j \bb\|^2_{L^2}  \big) + \frac{1}{2} 2^{2j\alpha} \| {\Delta}_j \bu\|^2_{L^2}+ \frac{1}{2} 2^j \|{\Delta}_j \bb\|^2_{L^2}={\sum_{i=1}^5} K_i,
  \end{split}
\end{equation}
where
\begin{equation*}
\begin{aligned}
& K_1=-\int_{\mathbb{R}^3} [{\Delta}_j, \bu\cdot \nabla]\bu \cdot {\Delta}_j \bu \mathrm{d}\x,   &K&_2=-\int_{\mathbb{R}^3} [{\Delta}_j, \bu\cdot \nabla]\bb \cdot {\Delta}_j \bb \mathrm{d}\x, \\
& K_3=\int_{\mathbb{R}^3} [{\Delta}_j, \bb \cdot \nabla]\bb \cdot {\Delta}_j \bu \mathrm{d}\x,   &K&_4=\int_{\mathbb{R}^3} [{\Delta}_j, \bb\cdot \nabla]\bu \cdot {\Delta}_j \bb \mathrm{d}\x, \\
& K_5=-\int_{\mathbb{R}^3} {\Delta}_j\left( \nabla \times \left( \left(\nabla \times \bb \right) \times \bb \right) \right) \cdot {\Delta}_j \bb \mathrm{d}\x.
  \end{aligned}
\end{equation*}
According to Lemma \ref{san}, it is straightforward to show that
\begin{align}
|K_1| & \lesssim \| {\Delta}_j \bu\|_{L^2} \|\nabla \bu\|_{L^\infty} \big(  \| {\Delta}_j \bu\|_{L^2} + {\sum_{k \geq j-1}} 2^{j-k}\| {{\Delta}}_k \bu\|_{L^2}  \big),\label{K1}\\[2mm]
|K_2| &\lesssim \| {\Delta}_j \bb\|^2_{L^2} \|\nabla \bu\|_{L^\infty} + C \| {\Delta}_j \bb\|_{L^2} \|\nabla \bb\|_{L^\infty}\| {\Delta}_j \bu\|_{L^2}\notag\\
& \quad +\| {\Delta}_j \bb\|_{L^2} \|\nabla \bu\|_{L^\infty} {\sum_{k \geq j-1}} 2^{j-k}\| {\Delta}_k \bb\|_{L^2},\label{K2}\\[2mm]
|K_3| &\lesssim \| {\Delta}_j \bu\|_{L^2} \|\nabla \bb\|_{L^\infty} \big(  \| {\Delta}_j \bb\|_{L^2} + {\sum_{k \geq j-1}} 2^{j-k}\| {\Delta}_k \bb\|_{L^2}  \big),\label{K3}\\[2mm]
|K_4| &\lesssim \| {\Delta}_j \bu\|_{L^2} \|\nabla \bb\|_{L^\infty} \|{\Delta}_j \bb\|_{L^2} + \|\nabla \bu\|_{L^\infty}\|{\Delta}_j \bb\|^2_{L^2}\notag\\
& \quad+ \|\nabla \bb\|_{L^\infty}\| {\Delta}_j \bb\|_{L^2} {\sum_{k \geq j-1}} 2^{j-k}\| {\Delta}_k \bu\|_{L^2}.\label{K4}
\end{align}
The rest of the poof is largely devoted to the estimate of $K_5$.

{\bf Step 3.} Using the fact $\bb \times ( \nabla \times \bb)=\frac{1}{2} \nabla(\bb \cdot \bb)-(\bb \cdot \nabla)\bb$, we write $K_5=K_5^{(1)}+K_5^{(2)}$, where
\begin{equation*}
\begin{split}
&K_5^{(1)}=-\int_{\mathbb{R}^3}[{\Delta}_j, \bb \cdot \nabla]\bb \cdot {\Delta}_j (\nabla \times \bb) \mathrm{d}\x,\\
&K_5^{(2)}=-\int_{\mathbb{R}^3} \left\{ {\Delta}_j \frac{1}{2}\nabla (\bb \cdot \bb)-(\bb \cdot \nabla) {\Delta}_j \bb \right\}  \cdot {\Delta}_j (\nabla \times \bb) \mathrm{d}\x.
\end{split}
\end{equation*}
Using Lemma \ref{san}, we first get
\begin{equation}\label{K50}
  |K_5^{(1)}| \lesssim 2^j \| \nabla \bb\|_{L^\infty}\| {\Delta}_j \bb\|_{L^2} \big(  \| {\Delta}_j \bb\|_{L^2} + {\sum_{k \geq j-1}} 2^{j-k}\| {\Delta}_k \bb\|_{L^2}  \big).
\end{equation}
By phase decomposition, we divide $K_5^{(2)}$ into
\begin{equation*}
\begin{split}
&K^{(2)}_{51}=\sum_{|k-j|\leq 2} \int_{\mathbb{R}^3} \big\{ {\Delta}_j(\nabla S_{k-1}\bb \cdot {\Delta}_k \bb)-\nabla {\Delta}_j S_{k-1}\bb \cdot {\Delta}_k \bb  \big\} \cdot {\Delta}_j (\nabla \times \bb) \mathrm{d}\x, \\
&K^{(2)}_{52}=\sum_{|k-j|\leq 2} \int_{\mathbb{R}^3} \big\{ {\Delta}_j( S_{k-1}\bb \cdot \nabla{\Delta}_k \bb)- S_{k-1}\bb \cdot \nabla {\Delta}_j {\Delta}_k \bb  \big\} \cdot {\Delta}_j (\nabla \times \bb) \mathrm{d}\x,\\
&K^{(2)}_{53}=\sum_{k\geq j-1} \int_{\mathbb{R}^3} \big\{ {\Delta}_j\nabla( \frac{1}{2} {\Delta}_k \bb \cdot {\tilde{\Delta}}_k \bb)- \nabla {\Delta}_j {\Delta}_k \bb   \cdot {\tilde{\Delta}}_k \bb \big\} \cdot {\Delta}_j (\nabla \times \bb) \mathrm{d}\x.
  \end{split}
\end{equation*}
By H\"older's inequality, we can show that
\begin{align}\label{K51}
|K^{(2)}_{51}| & \lesssim {\sum_{|k-j|\leq 2}} \| {\Delta}_j(\nabla S_{k-1}\bb \cdot {\Delta}_k \bb)-\nabla {\Delta}_j S_{k-1}\bb \cdot {\Delta}_k \bb \|_{L^2} \| {\Delta}_j (\nabla \times \bb)\|_{L^2}\nonumber\\
& \lesssim {\sum_{|k-j|\leq 2}} \| {\Delta}_j (\nabla \times \bb)\|_{L^2}  \left( \| \nabla S_{k-1} \bb \|_{L^\infty}\|{\Delta}_j {\Delta}_k \bb \|_{L^2}+ \| \nabla {\Delta}_j S_{k-1}\bb \|_{L^\infty} \| {\Delta}_k \bb \|_{L^2} \right)\nonumber
\\
& \lesssim 2^j \| \nabla \bb\|_{L^\infty} \|{\Delta}_j \bb \|^2_{L^2}.
\end{align}
By H\"older's inequality and commutator estimates, we have
\begin{align}\label{K52}
|K^{(2)}_{52} | &\leq {\sum_{|k-j|\leq 2}} \|[{\Delta}_j, S_{k-1} \bb \cdot \nabla]{\Delta}_k \bb\|_{L^2}\| {\Delta}_j (\nabla \times \bb)\|_{L^2}\nonumber
\\
& \lesssim {\sum_{|k-j|\leq 2}} \|\nabla S_{k-1} \bb\|_{L^\infty} \|{\Delta}_k \bb \|_{L^2}   \| {\Delta}_j (\nabla \times \bb)\|_{L^2}\nonumber
\\
& \lesssim 2^j \| \nabla \bb\|_{L^\infty} \|{\Delta}_j \bb \|^2_{L^2}.
\end{align}
For $K^{(2)}_{53}$, we calculate
\begin{equation*}
\begin{split}
K^{(2)}_{53}=&{\sum_{k\geq j-1}} \int_{\mathbb{R}^3} \big\{ {\Delta}_j\nabla( \frac{1}{2} {\Delta}_k \bb \cdot {\tilde{\Delta}}_k \bb)- \nabla {\Delta}_j {\Delta}_k \bb   \cdot {\tilde{\Delta}}_k \bb \big\} \cdot {\Delta}_j (\nabla \times \bb) \mathrm{d}\x
\\
=&\frac{1}{2}{\sum_{k\geq j-1}}  \int_{\mathbb{R}^3} \big\{ {\Delta}_j (\nabla {\Delta}_k \bb \cdot {\tilde{\Delta}}_k \bb)- \nabla {\Delta}_j {\Delta}_k \bb   \cdot {\tilde{\Delta}}_k \bb \big\} \cdot {\Delta}_j (\nabla \times \bb) \mathrm{d}\x
\\
&+\frac{1}{2}{\sum_{k\geq j-1}}  \int_{\mathbb{R}^3} \big\{ {\Delta}_j ( {\Delta}_k \bb \cdot \nabla{\tilde{\Delta}}_k \bb)- \nabla {\Delta}_j {\Delta}_k \bb   \cdot {\tilde{\Delta}}_k \bb \big\} \cdot {\Delta}_j (\nabla \times \bb) \mathrm{d}\x
\\
=&\frac{1}{2}{\sum_{k\geq j-1}}  \int_{\mathbb{R}^3} [{\Delta}_j, {\tilde{\Delta}}_k \bb \cdot \nabla] {\Delta}_k \bb\cdot {\Delta}_j (\nabla \times \bb) \mathrm{d}\x
\\
&+\frac{1}{2}{\sum_{k\geq j-1}}  \int_{\mathbb{R}^3} [{\Delta}_j, {{\Delta}}_k \bb \cdot \nabla] {\tilde{\Delta}}_k \bb \cdot {\Delta}_j (\nabla \times \bb) \mathrm{d}\x.
  \end{split}
\end{equation*}
Using H\"older's inequality, we can show that
\begin{align}\label{K532}
K^{(2)}_{53} \lesssim & {\sum_{k\geq j-1}}  \| [{\Delta}_j, {\tilde{\Delta}}_k \bb \cdot \nabla] {\Delta}_k \bb \|_{L^2} \| {\Delta}_j (\nabla \times \bb)\|_{L^2}\nonumber
\\
& + {\sum_{k\geq j-1}}  \| [{\Delta}_j, {{\Delta}}_k \bb \cdot \nabla] {\tilde{\Delta}}_k \bb \|_{L^2} \| {\Delta}_j (\nabla \times \bb)\|_{L^2}.
\end{align}
For the first term on the right-hand side of \eqref{K532}, using $\nabla \cdot {\tilde{\Delta}}_k \bb=0$ and Bernstein's inequality, we can show that
\begin{align}\label{P1}
  {\sum_{k \geq j-1}} \| [{\Delta}_j, {\tilde{\Delta}}_k \bb \cdot \nabla] {\Delta}_k \bb \|_{L^2} &\leq {\sum_{k \geq j-1}} \| \nabla \cdot [\Delta_j,{\tilde{\Delta}}_k \bb]{\Delta}_k \bb \|_{L^2}\nonumber
  \\
  & \lesssim {\sum_{k \geq j-1}} 2^j \|{\tilde{\Delta}}_k \bb \|_{L^\infty} \|{\Delta}_k \bb \|_{L^2}\nonumber
  \\
   & \lesssim  {\sum_{k \geq j-1}} 2^{j-k}\| \nabla \bb\|_{L^\infty} \|{\Delta}_k \bb \|_{L^2}.
\end{align}
Similarly, by using $\nabla \cdot {{\Delta}}_k \bb=0$ and Bernstein's inequality, we have
\begin{equation}\label{P2}
  {\sum_{k\geq j-1}}  \| [{\Delta}_j, {{\Delta}}_k \bb \cdot \nabla] {\tilde{\Delta}}_k \bb \|_{L^2} \lesssim  {\sum_{k \geq j-1}} 2^{j-k}\| \nabla \bb\|_{L^\infty} \|{\Delta}_k \bb \|_{L^2}.
\end{equation}
Substituting \eqref{P1} and \eqref{P2} into \eqref{K532}, we obtain
\begin{equation}\label{K53}
|K^{(2)}_{53}|\lesssim 2^j \| \nabla \bb\|_{L^\infty} \|{\Delta}_j \bb \|_{L^2}  {\sum_{k \geq j-1}} 2^{j-k}\|{\Delta}_k \bb \|_{L^2} .
\end{equation}
Combining \eqref{K51}, \eqref{K52} and \eqref{K53}, we obtain
\begin{equation}\label{K51a}
\begin{split}
| K^{(2)}_{5}| & \lesssim 2^j \| \nabla \bb\|_{L^\infty} \|{\Delta}_j \bb \|_{L^2} \big( \|{\Delta}_j \bb \|_{L^2} +{\sum_{k \geq j-1}} 2^{j-k}\|{\Delta}_k \bb \|_{L^2} \big).
\end{split}
\end{equation}
Combining \eqref{K51a} with \eqref{K50}, we conclude
\begin{equation}\label{K5}
 |K_{5}| \lesssim 2^j \| \nabla \bb\|_{L^\infty} \|{\Delta}_j \bb \|_{L^2}  \big( \|{\Delta}_j \bb \|_{L^2} +{\sum_{k \geq j-1}} 2^{j-k}\|{\Delta}_k \bb \|_{L^2} \big).
\end{equation}

{\bf Step 4.} Substituting the estimates of $K_1,...,K_5$, i.e., \eqref{K1}-\eqref{K4} and \eqref{K5}, into \eqref{000}, we have
\begin{align}\label{E9}
&\frac{\mathrm{d}}{\mathrm{d}t} \left( \| {\Delta}_j \bu\|^2_{L^2}+ \| {\Delta}_j \bb\|^2_{L^2}  \right) +  2^{2j\alpha} \| {\Delta}_j \bu\|^2_{L^2}+  2^j \| {\Delta}_j \bb\|^2_{L^2}\nonumber\\
\lesssim \ & \|(\nabla \bu, \nabla \bb) \|_{L^\infty}\big(  \| {\Delta}_j \bu\|^2_{L^2}+ \| {\Delta}_j \bb\|^2_{L^2}   \big)\nonumber\\
& + \|(\nabla \bu, \nabla \bb) \|_{L^\infty} \big[ \big( {\sum_{k \geq j-1}} 2^{j-k}\| {\Delta}_k \bu\|_{L^2}\big)^2 +  \big( {\sum_{k \geq j-1}} 2^{j-k}\| {\Delta}_k \bb\|_{L^2} \big)^2  \big]\nonumber\\
& + 2^j \|\nabla \bb \|_{L^\infty}  \| {\Delta}_j \bb\|^2_{L^2}+  2^j \|\nabla \bb \|_{L^\infty} \| {\Delta}_j \bb\|_{L^2} {\sum_{k \geq j-1}} 2^{j-k}\| {\Delta}_k \bb\|_{L^2}.
\end{align}
Multiplying \eqref{E9} by $2^{2sj}$, summing over $j \geq -1$, and integrating with respect to $t$, we deduce
\begin{equation*}
\begin{split}
K(t)  \lesssim \ & \| \Lambda^s \bu_0 \|^2_{L^2}+ \| \Lambda^s \bb_0 \|^2_{L^2}+  \int^t_0 \big( \|\nabla \bu\|_{L^\infty}+\|\nabla \bb\|_{L^\infty} \big) \| \Lambda^s\bu \|^2_{L^2}\mathrm{d}\tau\\
& +  {\sum_{j \geq -1}} 2^{(2s+1)j} \int^t_0 \|\nabla \bb\|_{L^\infty}\big(  \| {\Delta}_j \bb\|^2_{L^2}+( {\sum_{k \geq j-1}} 2^{j-k} \| {\Delta}_k \bb\|_{L^2}  )^2 \big) \mathrm{d}\tau\\
& + \int^t_0 \big( \|\nabla \bu\|_{L^\infty}+\|\nabla \bb\|_{L^\infty} \big) \| \Lambda^s \bb \|^2_{L^2} \mathrm{d}\tau,
\end{split}
\end{equation*}
where
\begin{equation*}
K(t)=\| \Lambda^s \bu(t) \|^2_{L^2}+ \| \Lambda^s \bb(t) \|^2_{L^2}+ \int^t_0 \| \Lambda^{\alpha+s} \bu(\tau) \|^2_{L^2} \mathrm{d} \tau+ \int^t_0 \| \Lambda^{\frac{1}{2}+s} \bb(\tau) \|^2_{L^2} \mathrm{d} \tau.
\end{equation*}
Using Young's inequality for series convolution, we obtain
\begin{equation*}
\begin{split}
{\sum_{j \geq -1}} 2^{(2s+1)j}  ( {\sum_{k \geq j-1}} 2^{j-k} \| {\Delta}_k \bb\|_{L^2}  )^2 &= {\sum_{j \geq -1}}  ( {\sum_{k \geq j-1}} 2^{(j-k)(s+\frac{3}{2})} 2^{(s+\frac{1}{2})k}\| {\Delta}_k \bb\|_{L^2}  )^2\\
& \lesssim {\sum_{j \geq -1}} 2^{(2s+1)j} \| {\Delta}_j \bb\|^2_{L^2}.
\end{split}
\end{equation*}
Consequently, we obtain
\begin{equation*}
\begin{split}
K(t) \lesssim & \ \| \Lambda^s \bu_0 \|^2_{L^2}+ \| \Lambda^s \bb_0 \|^2_{L^2}+  \int^t_0 \|\nabla \bb\|_{L^\infty} \| \Lambda^{\frac{1}{2}+s} \bb\|^2_{L^2} \mathrm{d}\tau\\
& + \int^t_0 \big( \|\nabla \bu\|_{L^\infty}+\|\nabla \bb\|_{L^\infty} \big) \big(\| \Lambda^s \bu\|^2_{L^2}+ \| \Lambda^s \bb \|^2_{L^2} \big)\mathrm{d}\tau.
  \end{split}
\end{equation*}
Using the Sobolev inequalities:
\begin{equation*}
\begin{split}
 \| \nabla \bb\|_{L^\infty} \lesssim \|\bb\|_{H^s}, \quad \| \nabla \bu\|_{L^\infty} \lesssim \|\bu\|_{H^s},\ s>\frac{5}{2},
 \end{split}
\end{equation*}
and the Gagliardo-Nirenberg inequalities:
\begin{align}\label{GN}
&\| \Lambda^{s} \bb \|_{L^2} \lesssim \|\Lambda^{\frac{1}{2}} \bb\|_{L^2}^{\frac{1}{2s}} \| \Lambda^{\frac{1}{2}+s} \bb\|^{1-\frac{1}{2s}}_{L^2},\nonumber\\
&\| \Lambda^{s} \bu \|_{L^2} \lesssim \|\Lambda^{\alpha} \bu\|_{L^2}^{\frac{\alpha}{s}} \| \Lambda^{\alpha+s} \bu\|^{1-\frac{\alpha}{s}}_{L^2},
\end{align}
we update the estimate of $K(t)$ as
\begin{align}\label{E1}
K(t) \lesssim \ & \| \Lambda^s \bu_0 \|^2_{L^2} + \| \Lambda^s \bb_0 \|^2_{L^2} + \int^t_0 \| \bb\|_{H^s}  \| \Lambda^{s+\frac{1}{2}} \bb \|^2_{L^2} \mathrm{d}\tau \nonumber\\
& + \int^t_0 \big( \|\bu\|_{H^s}+\| \bb\|_{H^s} \big)  \big(\|\Lambda^{\alpha} \bu\|_{L^2}^{\frac{2\alpha}{s}} \| \Lambda^{\alpha+s} \bu\|^{2-\frac{2\alpha}{s}}_{L^2}+\|\Lambda^{\frac{1}{2}} \bb\|_{L^2}^{\frac{1}{s}} \| \Lambda^{\frac{1}{2}+s} \bb\|^{2-\frac{1}{s}}_{L^2}\big) \mathrm{d}\tau.
\end{align}
Set
\begin{equation*}
  E(t)=\sup_{\tau \in [0,t]}\big(\| \bu(t) \|^2_{{H}^s}+ \| \bb(t) \|^2_{{H}^s}\big)+ \int^t_0 \|\Lambda^{\alpha} \bu(\tau) \|^2_{{H}^s} \mathrm{d} \tau + \int^t_0 \| \Lambda^{\frac{1}{2}} \bb(\tau) \|^2_{{H}^s} \mathrm{d} \tau.
\end{equation*}
Adding \eqref{E0} and \eqref{E1} implies $E(t) \lesssim E(0)+  E^\frac32(\tau)$, from which and standard continuation argument the global bound of $E(t)$ follows when $E(0)$ is sufficiently small. The uniqueness of the solution follows from a routine argument and the energy estimates derive above. We omit the details for brevity.
\subsection{Local existence with small energy}
After obtaining the global energy bounds, we now give a proof of local existence for small data. For \eqref{HMHD} is a quasilinear system, so we will seek a local solution in a weaker Sobolev spaces, and then improve it's regularity by combining the global bounds of energy. We will obtain the local existence through an approximation procedure, where the approximation procedure has been shown in \cite{CWW}.  For each positive integer $n$, we define
\begin{equation*}
  \widehat{J_n f}(\xi)=\chi_{B_n}(\xi)\hat{f}(\xi),
\end{equation*}
where $B_n$ denotes the closed ball of radius $n$ centered at $0$ and $\chi_{B_n}$ denotes the characteristic functions on ${B_n}$. For $a \geq 0$, we denote
\begin{equation*}
  H^a_{n}=\{f \in H^a(\mathbb{R}^3), \mathrm{supp}\hat{f} \subseteq B_n  \}.
\end{equation*}
We now seek a solution $(\bu,\bb)\in H^s_{n}$ satisfying
\begin{equation}\label{dd}
\begin{cases}
\bu_t+(-\Delta)^{\alpha} \bu+J_n ( J_n \mathcal{P} \bu \cdot \nabla J_n \mathcal{P} \bu) - J_n ( J_n \mathcal{P} \bb \cdot \nabla J_n \mathcal{P} \bb)=\mathbf{0},\\
\bb_t+(-\Delta)^{\beta} \bb+J_n( J_n \mathcal{P} \bu \cdot \nabla J_n \mathcal{P}\bb)-J_n( J_n \mathcal{P}  \bb \cdot \nabla J_n \mathcal{P}  \bu)
\\ \ \ \qquad \qquad \qquad \qquad \qquad+ J_n ( { \nabla \times \left( \left(\nabla \times J_n \mathcal{P} \bb \right) \times J_n \mathcal{P} \bb\right)}) =\mathbf{0},\\
\nabla\cdot  \bu =0, \quad \nabla\cdot \bb=0;\\
\bu|_{t=0}=J_n \bu_0, \ \ \bb|_{t=0}=J_n\bb_0,
\end{cases}
\end{equation}
where $\mathcal{P}$ denotes the projection onto divergence-free vector fields. For each fixed integer $n\geq 1$, it is not very hard, although tedious, to verify that the right hand side of \eqref{dd} satisfies the Lipschitz condition in $H^s_n$ and, by Picard's theorem, \eqref{dd}
has a unique global (in time) solution. The uniqueness implies that
\begin{equation*}
  J_n \mathcal{P} \bu=\bu, \quad J_n \mathcal{P} \bb= \bb.
\end{equation*}
and ensures the divergence-free conditions $\nabla\cdot  \bu = \nabla\cdot \bb=0$. Then, \eqref{dd} is
simplified to
\begin{equation}\label{dd0}
\begin{cases}
\bu_t+(-\Delta)^{\alpha} \bu+J_n (  \bu \cdot \nabla  \bu) - J_n (  \bb \cdot \nabla  \bb)=\mathbf{0},\\
\bb_t+(-\Delta)^{\frac12} \bb+J_n( \bu \cdot \nabla \bb)-J_n(  \bb \cdot \nabla   \bu)
+ J_n ( { \nabla \times \left( \left(\nabla \times  \bb \right) \times  \bb\right)}) =\mathbf{0},\\
\nabla\cdot  \bu =0, \quad \nabla\cdot \bb=0;\\
\bu|_{t=0}=J_n \bu_0, \ \ \bb|_{t=0}=J_n \bb_0,
\end{cases}
\end{equation}
We denote the solution of \eqref{dd0} as $(\bu^n, \bb^n)$. Set the total energy
\begin{equation*}
\begin{split}
  E^s(t)[\bu^n, \bb^n]=& \| \bu^n(t) \|^2_{{H}^s}+ \| \bb^n(t) \|^2_{{H}^s}
   + \int^t_0 ( \|\Lambda^{\alpha} \bu^n(\tau) \|^2_{{H}^s}  +  \| \Lambda^{\frac{1}{2}} \bb^n(\tau) \|^2_{{H}^s} ) \mathrm{d} \tau.
\end{split}
\end{equation*}
From section 3.1, we can get the following global bound
\begin{equation}\label{ub}
  E^s(1)[\bu^n, \bb^n] \lesssim \epsilon.
\end{equation}
We then find a limit of the sequence $\{(\bu^n, \bb^n)\}_{n\geq 1}$ in a appropriate functional spaces. To do that, let us now estimate the lower-order energy
\begin{equation*}
  E^{0}(1)[\bu^n-\bu^m, \bb^n-\bb^m].
\end{equation*}
By \eqref{dd0}, we only need to see the quasi-linear part
\begin{equation*}
  \begin{split}
   Q=& \int_{\mathbb{R}^3}  \left( \nabla \times \big( (\nabla \times  \bb^n ) \times  \bb^n \big)-\nabla \times \big( (\nabla \times  \bb^m ) \times  \bb^m \big) \right) \cdot ( \bb^n- \bb^m)dx
  \\
  =& \int_{\mathbb{R}^3}   \nabla \times \big( (\nabla \times  (\bb^n-\bb^m) ) \times  \bb^n \big)  \cdot ( \bb^n- \bb^m)dx
  \\
  & + \int_{\mathbb{R}^3}   \nabla \times \big( (\nabla \times  \bb^m ) \times  (\bb^n-\bb^m) \big)  \cdot ( \bb^n- \bb^m)dx
  \\
  =& -\int_{\mathbb{R}^3}   \big( (\nabla \times  (\bb^n-\bb^m) ) \times  \bb^n \big)  \cdot  \nabla \times ( \bb^n- \bb^m)dx
  \\
  & + \int_{\mathbb{R}^3}   \nabla \times \big( (\nabla \times  \bb^m ) \times  (\bb^n-\bb^m) \big)  \cdot ( \bb^n- \bb^m)dx
  \\
  =& \int_{\mathbb{R}^3}   \nabla \times \big( (\nabla \times  \bb^m ) \times  (\bb^n-\bb^m) \big)  \cdot ( \bb^n- \bb^m)dx
  \\
  =& \int_{\mathbb{R}^3}  \left\{  (\bb^n-\bb^m) \cdot \nabla (\nabla \times \bb^m)-  (\nabla \times \bb^m ) \cdot \nabla ( \bb^n- \bb^m)\right\}\cdot ( \bb^n- \bb^m)dx.
  \end{split}
\end{equation*}
Using H\"older's inequality, we can derive that
\begin{equation*}
  \begin{split}
  |Q| \lesssim & \| \bb^n-\bb^m \|_{L^3} \|\nabla (\nabla \times \bb^m)\|_{L^6} \| \bb^n- \bb^m \|_{L^2}
  \\
  & + \|\nabla \times \bb^m\|_{L^4} \|\nabla ( \bb^n- \bb^m)\|_{L^4} \| \bb^n- \bb^m \|_{L^2}
  \\
  \lesssim & (\|\Lambda^{\frac12} \bb^m\|^2_{H^s}+\|\Lambda^{\frac12} \bb^n\|^2_{H^s}) \| \bb^n- \bb^m \|^2_{L^2}+ \frac{1}{100} \|\Lambda^{\frac12} (\bb^n-\bb^m)\|^2_{L^2}.
  \end{split}
\end{equation*}
These uniform bounds allow us to show that
\begin{equation*}
  \frac{d}{dt} \| \bb^n-\bb^m \|^2_{L^2}+\frac{1}{4}\|\Lambda^{\frac12} (\bb^n-\bb^m) \|^2_{L^2}  \leq C(\|\Lambda^{\frac12} \bb^m\|^2_{H^s}+\|\Lambda^{\frac12} \bb^n\|^2_{H^s}) \| \bb^n- \bb^m \|^2_{L^2}.
\end{equation*}
Integrating the above inequality on $[0,t]$, we can deduce that by Gronwall's inequality
\begin{equation*}
  E^{0}(1)[\bu^n-\bu^m, \bb^n-\bb^m] \leq CE^{0}(0)[\bu^n-\bu^m, \bb^n-\bb^m].
\end{equation*}
As a result, we have
\begin{equation*}
  E^{0}(1)[\bu^n-\bu^m, \bb^n-\bb^m]\rightarrow 0, \quad \mathrm{if} \ n, m \rightarrow \infty.
\end{equation*}
Let $(\bu,\bb)$ be the limit. Then $E^s(1)[\bu,\bb] \lesssim \epsilon.$ By \eqref{ub} and the interpolation inequality
\begin{equation*}
  \| f \|_{H^a} \leq \| f \|^{1-\frac{a}{s}}_{L^2}\| f \|^{\frac{a}{s}}_{H^s}, \quad 0<a<s,
\end{equation*}
we further obtain the strong convergence
\begin{equation*}
  E^{a}(1)[\bu^n-\bu,\bb^n-\bb] \rightarrow 0, n \rightarrow \infty, \quad 0<a<s.
\end{equation*}
Consequently, $(\bu,\bb) \in C([0,1]; H^a)$. This strong convergence makes it easy to
check that $(\bu,\bb)$ satisfies the Hall-MHD equation \eqref{HMHD}. Combining the local existence and global bounds in Section 3.2 and Section 3.1 respectively, we can conclude the proof of Theorem \ref{thm1}. \hfill$\square$

\section{Proof of Theorem \ref{thm2}}
In this section, we prove Theorem \ref{thm2} by using Littlewood-Paley decomposition and energy method. Let $\bU$ and $\bB$ be solutions to \eqref{U} and \eqref{B}, respectively. Let $\bu=\Bf+{\bU}, \bb=\Bh+{\bB}$, where $(\bu,\bb)$ denotes the solution to \eqref{HMHD}. We now give a global bound of the energy estimates for the disturbed velocity $\Bf$ and the disturbed magnetic field $\Bh$. By using \eqref{HMHD}, \eqref{U} and \eqref{B}, then $\Bf$ and $\Bh$ satisfy
\begin{equation}\label{fh}
\left\{
\begin{aligned}
\Bf_t+ (-\Delta)^\alpha \Bf+\nabla \tilde{P}&=-\Bf \cdot \nabla \Bf-{\bU} \cdot \nabla \Bf - \Bf \cdot \nabla {\bU}+\Bh \cdot \nabla \Bh + {\bB} \cdot \nabla \Bh+\Bh \cdot \nabla {\bB}+\bF,\\
\Bh_t+ (-\Delta)^{\frac{1}{2}} \Bh&=-\Bf \cdot \nabla \Bh-{\bU} \cdot \nabla \Bh - \Bf \cdot \nabla {\bB}+\Bh \cdot \nabla \Bf+ {\bB} \cdot \nabla \Bf + \Bh \cdot \nabla {\bU}\\
&\quad-\nabla \times \left( \left(\nabla \times \bb \right) \times \bb\right)+\bG,
\end{aligned}
\right.
\end{equation}
where $\tilde{P}=P+\frac{1}{2}|{\bU}|^2-\frac{1}{2}|{\bB}|^2$, $\bF= {\bU}\times (\nabla \times {\bU})-{\bB}\times \left(\nabla \times {\bB}\right)$ and $\bG= \nabla \times ( {\bU} \times {\bB})$. Note that according to Lemma \ref{L1}, $\bF$ and $\bG$ satisfy the estimates \eqref{U1} and \eqref{UB}, respectively. We divide the subsequent proof into several steps.

\textbf{Step 1.} Taking $L^2$ inner product of the two equations in \eqref{fh} with $\Bf$ and $\Bh$, respectively, we have
\begin{align}\label{z1}
&\|\Bf(t)\|^2_{L^2}+\|\Bh(t)\|^2_{L^2}+ \int^t_0 \|\Lambda^{\alpha} \Bf(\tau)\|^2_{L^2}\mathrm{d}\tau+ \int^t_0 \|\Lambda^{\frac{1}{2}} \Bh(\tau)\|^2_{L^2}\mathrm{d}\tau\nonumber\\
\leq \ & \|\bu_{01}\|^2_{L^2}+\|\bb_{01}\|^2_{L^2}+ \int^t_0 \|(\nabla \bU, \nabla \bB)  \|_{L^\infty}\left( \|\Bf\|^2_{L^2}+ \|\Bh\|^2_{L^2} \right)\mathrm{d}\tau\nonumber\\
& -\int^t_0 \int_{\mathbb{R}^3}  \nabla \times \left( (\nabla \times \bb) \times \bb  \right) \cdot \Bh\, \mathrm{d}x \mathrm{d}\tau+ \int^t_0 \left( \| \bF \|_{L^2} \| \Bf \|_{L^2}+\| \bG \|_{L^2} \| \Bh\|_{L^2} \right) \mathrm{d}\tau,
\end{align}
where we used the decomposition $\bu_0=\bu_{01}+\bu_{02},\ \bb_0=\bb_{01}+\bb_{02}$ (see Theorem \ref{thm1}) and the initial condition $\bU_0=\bu_{02},\ \bB_0=\bb_{02}$ (see \eqref{U} and \eqref{B}). Since $\bb=\bB+\Bh$, we can write $\int^t_0  \int_{\mathbb{R}^3} \nabla \times \left( (\nabla \times \bb) \times \bb  \right) \cdot \Bh \,\mathrm{d}x\mathrm{d}\tau$ as
\begin{equation*}
\begin{split}
  &\int^t_0  \int_{\mathbb{R}^3} \nabla \times \left( (\nabla \times \bb) \times \bb  \right) \cdot \Bh\, \mathrm{d}x \mathrm{d}\tau
  =\int^t_0 \int_{\mathbb{R}^3} \left( (\nabla \times \bb) \times \bb  \right) \cdot \nabla \times \Bh\, \mathrm{d}x \mathrm{d}\tau
  \\
  =&\int^t_0  \int_{\mathbb{R}^3} \left( (\nabla \times \bB) \times \bB  \right) \cdot \nabla \times \Bh\, \mathrm{d}x \mathrm{d}\tau+\int^t_0 \int_{\mathbb{R}^3}  \left( (\nabla \times \bB) \times \Bh  \right) \cdot \nabla \times \Bh\,  \mathrm{d}x \mathrm{d}\tau.
  \end{split}
\end{equation*}
Using H\"older's inequality, we can show that for $s>\frac52$,
$$
\begin{aligned}
&\left| \int^t_0  \int_{\mathbb{R}^3} \left( (\nabla \times \bB) \times \bB  \right) \cdot \nabla \times \Bh\, \mathrm{d}x \mathrm{d}\tau+\int^t_0 \int_{\mathbb{R}^3}  \left( (\nabla \times \bB) \times \Bh  \right) \cdot \nabla \times \Bh\, \mathrm{d}x \mathrm{d}\tau \right| \\
\le & \int^t_0 \| (\nabla \times \bB) \times \bB \|_{L^2} \|\Bh\|_{H^s} \mathrm{d}\tau+\int^t_0 \| \nabla \bB \|_{L^\infty} \|\Bh\|^2_{H^s} \mathrm{d}\tau.
\end{aligned}
$$
So we update \eqref{z1} as
\begin{align}\label{L2}
&\|\Bf(t)\|^2_{L^2}+\|\Bh(t)\|^2_{L^2}+ \int^t_0 \|\Lambda^{\alpha} \Bf(\tau)\|^2_{L^2}\mathrm{d}\tau+ \int^t_0 \|\Lambda^{\frac{1}{2}} \Bh(\tau)\|^2_{L^2}\mathrm{d}\tau\nonumber\\
\leq \ & \|\bu_{01}\|^2_{L^2}+\|\bb_{01}\|^2_{L^2}+ \int^t_0 \|(\nabla \bU, \nabla \bB)  \|_{L^\infty}\left( \|\Bf\|^2_{L^2}+ \|\Bh\|^2_{L^2} \right)\mathrm{d}\tau\nonumber\\
& +\int^t_0 \| (\nabla \times \bB) \times \bB \|_{L^2} \|\Bh\|_{H^s} \mathrm{d}\tau+\int^t_0 \| \nabla \bB \|_{L^\infty} \|\Bh\|^2_{H^s} \mathrm{d}\tau\nonumber\\
&+ \int^t_0 \left( \| \bF \|_{L^2} \| \Bf \|_{L^2}+\| \bG \|_{L^2} \| \Bh\|_{L^2} \right) \mathrm{d}\tau.
\end{align}

\textbf{Step 2.}
For $j\ge-1$, taking ${\Delta}_j$ to \eqref{fh}, we have
\begin{equation*}
\left\{
\begin{aligned}
&{\partial}_t{\Delta}_j \Bf+{\Delta}_j (-\Delta)^{\alpha} \Bf + {\Delta}_j(\Bf \cdot \nabla \Bf) + \nabla {\Delta} _j \tilde{P}= {\Delta}_j(\Bh\cdot \nabla \Bh) -{\Delta}_j({\bU} \cdot \nabla \Bf) \\
&\qquad - {\Delta}_j(\Bf \cdot \nabla {\bU})+{\Delta}_j( {\bB} \cdot \nabla \Bh) +{\Delta}_j (\Bh \cdot \nabla {\bB})+ {\Delta}_j \bF,\\
&{\partial}_t {\Delta}_j \Bh + {\Delta}_j(\Bf \cdot \nabla \Bh) + {\Delta}_j (-\Delta)^{\frac{1}{2}}\Bh=  {\Delta}_j(\Bh \cdot \nabla \Bf)-{\Delta}_j( {\bU} \cdot \nabla \Bh ) \\
& \qquad - {\Delta}_j (\Bf \cdot \nabla {\bB})+ {\Delta}_j ({\bB} \cdot \nabla \Bf)+ {\Delta}_j(\Bh \cdot \nabla {\bU}) - {\Delta}_j \left( \nabla \times \left( \left(\nabla \times \bb \right) \times \bb\right) \right)+{\Delta}_j \bG.
\end{aligned}
\right.
\end{equation*}
Taking the $L^2$ inner product of the above equations with ${\Delta}_j \Bf$ and ${\Delta}_j \Bh$, respectively, we obtain
\begin{align}\label{fh}
 &\frac{1}{2} \frac{\mathrm{d}}{\mathrm{d}t} ( \|{\Delta}_j \Bf\|^2_{L^2}+ \|{\Delta}_j \Bh\|^2_{L^2}) +  2^{2j\alpha} \|{\Delta}_j \Bf\|^2_{L^2} + 2^{j} \|{\Delta}_j \Bh\|^2_{L^2}\nonumber\\
=\ & {\sum_{k=0}^{12}} J_k+ \int_{\mathbb{R}^3} {\Delta}_j \bF \cdot {\Delta}_j \Bf\, \mathrm{d}\x+ \int_{\mathbb{R}^3} {\Delta}_j \bG \cdot {\Delta}_j \Bh\, \mathrm{d}\x ,
\end{align}
where\begin{align*}
& J_0=-\int_{\mathbb{R}^3} {\Delta}_j(\Bf \cdot \nabla \Bf)\cdot {\Delta}_j \Bf\, \mathrm{d}\x, &J&_1=\int_{\mathbb{R}^3} {\Delta}_j(\Bh\cdot \nabla \Bh) \cdot {\Delta}_j \Bf\, \mathrm{d}\x,\\
& J_2= -\int_{\mathbb{R}^3} {\Delta}_j({\bU} \cdot \nabla \Bf) \cdot {\Delta}_j \Bf\, \mathrm{d}\x, &J&_3= -\int_{\mathbb{R}^3} {\Delta}_j(\Bf \cdot \nabla {\bU})\cdot {\Delta}_j \Bf\, \mathrm{d}\x,\\
& J_4= \int_{\mathbb{R}^3} {\Delta}_j( {\bB} \cdot \nabla \Bh) \cdot {\Delta}_j \Bf\, \mathrm{d}\x, &J&_5 =\int_{\mathbb{R}^3} {\Delta}_j (\Bh \cdot \nabla {\bB})\cdot {\Delta}_j \Bf\, \mathrm{d}\x,\\
&J_6=-\int_{\mathbb{R}^3} {\Delta}_j(\Bf \cdot \nabla \Bh)\cdot {\Delta}_j \Bh\, \mathrm{d}\x,
&J&_7= \int_{\mathbb{R}^3} {\Delta}_j(\Bh \cdot \nabla \Bf)\cdot {\Delta}_j \Bh\, \mathrm{d}\x,\\
&J_8=-\int_{\mathbb{R}^3}  {\Delta}_j( {\bU} \cdot \nabla \Bh )\cdot {\Delta}_j \Bh\, \mathrm{d}\x, &J&_{9}=- \int_{\mathbb{R}^3} {\Delta}_j (\Bf \cdot \nabla {\bB})\cdot {\Delta}_j \Bh\, \mathrm{d}\x,\\
&J_{10}=\int_{\mathbb{R}^3} {\Delta}_j ({\bB} \cdot \nabla \Bf) \cdot {\Delta}_j \Bh\, \mathrm{d}\x, &J&_{11}= \int_{\mathbb{R}^3} {\Delta}_j(\Bh \cdot \nabla {\bU})\cdot {\Delta}_j \Bh\, \mathrm{d}\x,\\
& J_{12}=-\int_{\mathbb{R}^3} {\Delta}_j\left( \nabla \times \left( \left(\nabla \times \bb \right) \times \bb \right) \right) \cdot {\Delta}_j \Bh\, \mathrm{d}\x.
\end{align*}	
Next, we estimate the $J_k$'s one by one. The estimates are largely based on Lemma \ref{san}.

{\bf Step 3.} Using the incompressibility condition and Lemma \ref{san}, we can show that
\begin{align}\label{J0}
|J_0|&= \left|  \int_{\mathbb{R}^3}  [{\Delta}_j, \Bf\cdot \nabla]\Bf \cdot {\Delta}_j \Bf\, \mathrm{d}\x \right| \nonumber\\
&
\lesssim \| {\Delta}_j \Bf\|_{L^2} \|\nabla \Bf\|_{L^\infty} \big(  \| {\Delta}_j \Bf\|_{L^2} + {\sum_{k \geq j-1} 2^{j-k}}\| {\Delta}_k \Bf\|_{L^2}  \big).
\end{align}
Noticing
\begin{equation*}
  J_1+J_7=\int_{\mathbb{R}^3} [{\Delta}_j, \Bh \cdot \nabla]\Bh \cdot {\Delta}_j \Bf\, \mathrm{d}\x+ \int_{\mathbb{R}^3} [{\Delta}_j, \Bh \cdot \nabla]\Bf \cdot {\Delta}_j \Bh\, \mathrm{d}\x,
\end{equation*}
we can show by using Lemma \ref{san} that
\begin{align}\label{J1}
  |J_1+J_7| & \lesssim \| {\Delta}_j \Bf\|_{L^2} \|\nabla \Bh\|_{L^\infty} \big(  \| {\Delta}_j \Bf\|_{L^2} + {\sum_{k \geq j-1}} 2^{j-k}\| {\Delta}_k \Bf\|_{L^2}  \big)\nonumber\\
  &\quad+ \| {\Delta}_j \Bh\|_{L^2} \|\nabla \Bf\|_{L^\infty} \|{\Delta}_j \Bh\|_{L^2} + \|\nabla \Bh\|_{L^\infty}\|{\Delta}_j \Bf\|_{L^2}\|{\Delta}_j \Bh\|_{L^2}\nonumber\\
  & \quad+ \|\nabla \Bh\|_{L^\infty}\| {\Delta}_j \Bh\|_{L^2} {\sum_{k \geq j-1}} 2^{j-k}\| {\Delta}_k \Bf\|_{L^2}.
\end{align}
Similarly, we can show that
\begin{align}\label{J2}
|J_2|&= \left|\int_{\mathbb{R}^3} [\Delta_j , {\bU} \cdot \nabla ] \Bf \cdot {\Delta}_j \Bf\, \mathrm{d}\x  \right|\nonumber\\
& \lesssim  \|\nabla \bU\|_{L^\infty} \|{\Delta}_j \Bf\|^2_{L^2}+\|{\Delta}_j \bU\|_{L^\infty} \|\nabla \Bf\|_{L^2}\|{\Delta}_j \Bf\|_{L^2} \nonumber \\
& \quad +\|\nabla \bU\|_{L^\infty}\|{\Delta}_j \Bf\|_{L^2}{\sum_{k \geq j-1}} 2^{j-k} \|{{\Delta}}_k \Bf\|_{L^2}.
\end{align}
It follows directly from H\"older's inequality that
\begin{align}\label{J3}
  |J_3| \lesssim \| {\Delta}_j (\Bf \cdot \nabla \bU) \|_{L^2} \|{\Delta}_j \Bf\|_{L^2}, \quad |J_5| \lesssim \| {\Delta}_j (\Bh \cdot \nabla \bB) \|_{L^2} \|{\Delta}_j \Bf\|_{L^2},
\end{align}
and
\begin{equation}\label{J5}
  |J_9| \lesssim \| {\Delta}_j (\Bf \cdot \nabla \bB) \|_{L^2} \|{\Delta}_j \Bh\|_{L^2},  \quad |J_{11}| \lesssim \| {\Delta}_j (\Bh \cdot \nabla \bU) \|_{L^2} \|{\Delta}_j \Bh\|_{L^2}.
\end{equation}
A direct calculation shows that
\begin{equation*}
 J_4+J_{10}= \int_{\mathbb{R}^3} [{\Delta}_j, {\bB} \cdot \nabla ] \Bh \cdot {\Delta}_j \Bf\, \mathrm{d}\x+\int_{\mathbb{R}^3} [{\Delta}_j, {\bB} \cdot \nabla ] \Bf \cdot {\Delta}_j \Bh\, \mathrm{d}\x.
\end{equation*}
Using Lemma \ref{san}, we can show that
\begin{align}\label{J4}
|J_4+J_{10}| \lesssim \ &\|\nabla \bB\|_{L^\infty} (\|{\Delta}_j \Bh\|^2_{L^2}+\|{\Delta}_j \Bf\|^2_{L^2})\nonumber\\
&+\|{\Delta}_j \bB\|_{L^\infty} \|\nabla \Bh\|_{L^2}\|{\Delta}_j \Bf\|_{L^2}+\|{\Delta}_j \bB\|_{L^\infty} \|\nabla \Bf\|_{L^2}\|{\Delta}_j \Bf\|_{L^2}\nonumber\\
& +{\sum_{k \geq j-1}}\|\nabla \bB\|_{L^\infty}( \|{\Delta}_j \Bf\|_{L^2} \|{{\Delta}}_k \Bh\|_{L^2} + \|{\Delta}_j \Bh\|_{L^2}\|{{\Delta}}_k \Bf\|_{L^2}).
\end{align}
Repeated application of Lemma \ref{san} shows that
\begin{align}\label{J6}
|J_6|&= \big| \int_{\mathbb{R}^3} [\Delta_j, {\Bf} \cdot \nabla] \Bh \cdot {\Delta}_j \Bh\, \mathrm{d}\x \big|\nonumber\\
&\lesssim \| {\Delta}_j \Bh\|^2_{L^2} \|\nabla \Bf\|_{L^\infty} +  \| {\Delta}_j \Bh\|_{L^2} \|\nabla \Bh\|_{L^\infty}\| {\Delta}_j \Bf\|_{L^2}\nonumber\\
& \quad +\| {\Delta}_j \Bh\|_{L^2} \|\nabla \Bf\|_{L^\infty}{\sum_{k \geq j-1}} 2^{j-k}\| {\Delta}_k\Bh\|_{L^2},
\end{align}
and
\begin{align}\label{J8}
  |J_8| &= \left| \int_{\mathbb{R}^3} [{\Delta}_j, \bU \cdot \nabla]\Bh \cdot {\Delta}_j \Bh\,\mathrm{d}\x \right|\nonumber\\
  & \lesssim  \|\nabla \bB\|_{L^\infty} \|{\Delta}_j \Bh\|_{L^2}\|{\Delta}_j \Bf\|_{L^2}+\|{\Delta}_j \bB\|_{L^\infty} \|\nabla \Bh\|_{L^2}\|{\Delta}_j \Bf\|_{L^2}\nonumber \\
  & \quad  +\|\nabla \bB\|_{L^\infty}\|{\Delta}_j \Bf\|_{L^2}{\sum_{k \geq j-1}} 2^{j-k} \|{{\Delta}}_k \Bh\|_{L^2}.
\end{align}
It remains to estimate $J_{12}$. Since $\bb={\bB}+\Bh$, we rewrite $J_{12}$ as
\begin{equation*}
  J_{12}=J^{(1)}_{12}+J^{(2)}_{12}+J^{(3)}_{12}+J^{(4)}_{12},
\end{equation*}
where
\begin{equation*}
  \begin{split}
 & J^{(1)}_{12}=\int_{\mathbb{R}^3} {\Delta}_j \left( \big(\nabla \times {\bB} \big) \times {\bB} \right)  \cdot {\Delta}_j \nabla \times \Bh\, \mathrm{d}\x,\\
  & J^{(2)}_{12}=\int_{\mathbb{R}^3} {\Delta}_j \left( \big(\nabla \times \Bh \big) \times {\bB} \right)  \cdot {\Delta}_j \nabla \times \Bh\, \mathrm{d}\x,\\
  & J^{(3)}_{12}=\int_{\mathbb{R}^3} {\Delta}_j \left( \big(\nabla \times {\bB} \big) \times \Bh \right)  \cdot {\Delta}_j \nabla \times \Bh\, \mathrm{d}\x,\\
  & J^{(4)}_{12}=\int_{\mathbb{R}^3} {\Delta}_j \left( \big(\nabla \times \Bh \big) \times \Bh \right)  \cdot {\Delta}_j \nabla \times \Bh\, \mathrm{d}\x.
  \end{split}
\end{equation*}
Using H\"older's inequality, we can show that
\begin{equation*}
\begin{split}
  |J^{(1)}_{12}| &\lesssim \|{\Delta}_j \left( \big(\nabla \times {\bB} \big) \times {\bB} \right)\|_{L^2} \|{\Delta}_j \nabla \times \Bh\|_{L^2}\\
  & \lesssim 2^{\frac{j}{2}}\|{\Delta}_j \left( \big(\nabla \times {\bB} \big) \times {\bB} \right)\|_{L^2} \cdot 2^{\frac{j}{2}}\|{\Delta}_j \Bh\|_{L^2},\\
  |J^{(3)}_{12}| &\lesssim \|{\Delta}_j \left( \big(\nabla \times {\bB} \big) \times \Bh \right)\|_{L^2} \|{\Delta}_j \nabla \times \Bh\|_{L^2}\\
  & \lesssim 2^{\frac{j}{2}}\|{\Delta}_j \left( \big(\nabla \times {\bB} \big) \times \Bh \right)\|_{L^2} \cdot 2^{\frac{j}{2}}\|{\Delta}_j \Bh\|_{L^2}.
\end{split}
\end{equation*}
Since $ \big( {\Delta}_j(\nabla \times \Bh ) \times {\bB}  \big)  \cdot {\Delta}_j (\nabla \times \Bh)=0$, we derive that
\begin{equation*}
\begin{split}
  J^{(2)}_{12}=\int_{\mathbb{R}^3} \left( {\Delta}_j \big( \big(\nabla \times \Bh \big) \times {\bB}\big)- {\Delta}_j\big(\nabla \times \Bh \big) \times {\bB}  \right)  \cdot {\Delta}_j \nabla \times \Bh\, \mathrm{d}\x.
\end{split}
\end{equation*}
Using H\"older's inequality and commutator estimates, we can show that
\begin{equation*}
\begin{split}
|J^{(2)}_{12}| \lesssim \|{\Delta}_j \bB\|_{L^\infty} \|\nabla \Bh\|_{L^2} \|{\Delta}_j \nabla \times \Bh \|_{L^2}.
\end{split}
\end{equation*}
According to Lemma \ref{san}, we have
\begin{equation*}
  |J^{(4)}_{12}| \lesssim 2^j \|\nabla \Bh\|_{L^\infty} \|{\Delta}_j \Bh \|_{L^2} \cdot \big( \|{\Delta}_j \Bh \|_{L^2} +{\sum_{k \geq j-1}} 2^{j-k}\|{\Delta}_k \Bh \|_{L^2} \big).
\end{equation*}
We then conclude that
\begin{align}\label{J12}
  |J_{12}|
   \lesssim \ &2^{\frac{j}{2}}(\|{\Delta}_j ( (\nabla \times {\bB} ) \times {\bB} )\|_{L^2}+\|{\Delta}_j ( (\nabla \times {\bB} ) \times \Bh )\|_{L^2}) \cdot 2^{\frac{j}{2}}\|{\Delta}_j \Bh\|_{L^2}\nonumber \\
   &+\|{\Delta}_j \bB\|_{L^\infty} \|\nabla \Bh\|_{L^2} \|{\Delta}_j \nabla \times \Bh \|_{L^2}\nonumber\\
  &+2^j \|\nabla \Bh\|_{L^\infty} \|{\Delta}_j \Bh \|_{L^2} \cdot ( \|{\Delta}_j \Bh \|_{L^2} +{\sum_{k \geq j-1}} 2^{j-k}\|{\Delta}_k \Bh \|_{L^2} ).
\end{align}

{\bf Step 4.} After a series of operations (multiplying the resulting inequality by $2^{2sj}$, summing over $j \geq -1$, and integrating with respect to time) we can show from \eqref{fh} that
\begin{align}\label{highe}
 &\frac{1}{2}  \big( {\sum_{j \geq -1}}  2^{2sj}\|{\Delta}_j \Bf\|^2_{L^2}+{\sum_{j \geq -1}} 2^{2sj}\|{\Delta}_j \Bh\|^2_{L^2} \big) +  \int^t_0\big( {\sum_{j \geq -1}} 2^{2j\alpha+2sj} \|{\Delta}_j \Bf\|^2_{L^2} + 2^{j+2sj} \|{\Delta}_j \Bh\|^2_{L^2} \big)\mathrm{d}\tau \nonumber \\
=\ & {\sum_{j \geq -1}} \left( \| 2^{js} {\Delta}_j \bu_{01} \|^2_{L^2}+ \| 2^{js} {\Delta}_j \bb_{01} \|^2_{L^2}\right) + \int^t_0 {\sum_{j \geq -1}}2^{2sj} {\sum_{k=0}^{12}} J_k \mathrm{d}\tau\nonumber \\
&+ \int^t_0 \int_{\mathbb{R}^3} {\sum_{j \geq -1}} 2^{sj}{\Delta}_j \bF \cdot 2^{sj}{\Delta}_j \Bf\, \mathrm{d}\x \mathrm{d}\tau +  \int^t_0 \int_{\mathbb{R}^3} {\sum_{j \geq -1}} 2^{sj}{\Delta}_j \bG \cdot 2^{sj}{\Delta}_j \Bh\, \mathrm{d}\x \mathrm{d}\tau .
\end{align}
Using definition of the $\dot{H}^s$ norm, we update \eqref{highe} as
\begin{align}\label{highe2}
 & \|  \Bf(t) \|^2_{\dot{H}^s}+ \| \Lambda^s \Bh(t) \|^2_{\dot{H}^s}+ \int^t_0 \| \Lambda^{\alpha} \Bf(\tau) \|^2_{\dot{H}^s} \mathrm{d} \tau+  \int^t_0 \| \Lambda^{\frac{1}{2}} \Bh(\tau) \|^2_{\dot{H}^s} \mathrm{d} \tau\nonumber\\
  \lesssim \ & \|  \bu_{01} \|^2_{\dot{H}^s}+ \| \bb_{01} \|^2_{\dot{H}^s}+ \int^t_0 {\sum_{j \geq -1}}2^{2sj} {\sum_{k=0}^{12}} J_k \mathrm{d}\tau\nonumber
\\
&+ \int^t_0 \int_{\mathbb{R}^3} {\sum_{j \geq -1}} 2^{sj}{\Delta}_j \bF \cdot 2^{sj}{\Delta}_j \Bf\, \mathrm{d}\x \mathrm{d}\tau +  \int^t_0 \int_{\mathbb{R}^3} {\sum_{j \geq -1}} 2^{sj}{\Delta}_j \bG \cdot 2^{sj}{\Delta}_j \Bh\, \mathrm{d}\x \mathrm{d}\tau .
\end{align}
Using H\"older's inequality, we can estimate the second line on the right-hand side of \eqref{highe2} as
\begin{align}\label{highe3}
  & \int^t_0 \int_{\mathbb{R}^3} {\sum_{j \geq -1}} 2^{sj}{\Delta}_j \bF \cdot 2^{sj}{\Delta}_j \Bf\, \mathrm{d}\x \mathrm{d}\tau  +  \int^t_0 \int_{\mathbb{R}^3} {\sum_{j \geq -1}} 2^{sj}{\Delta}_j \bG \cdot 2^{sj}{\Delta}_j \Bh\, \mathrm{d}\x \mathrm{d}\tau\nonumber
\\
\leq & \int^t_0 \left( \| \bF \|_{\dot{H}^s} \| \Bf \|_{\dot{H}^s}+ \| \bG \|_{\dot{H}^s} \| \Bh \|_{\dot{H}^s}   \right)  \mathrm{d}\tau.
\end{align}
In the next step, we estimate the second term in the first line on the right-hand side of \eqref{highe2}, by using the estimates derived in {\bf Step 3}.

{\bf Step 5.} Using \eqref{J0} and H\"older's inequality, we can show that
\begin{align}\label{EJ0}
  & \Big| \int^t_0 {\sum_{j \geq -1}}2^{2sj}  J_0 \mathrm{d}\tau  \Big|\nonumber
  \\
  \lesssim & \int^t_0 \|\nabla \Bf\|_{L^\infty} \|  \Bf\|^2_{\dot{H}^s}\mathrm{d}\tau+ \int^t_0 \|\nabla \Bf\|_{L^\infty} {\sum_{j \geq -1}} \big( 2^{sj}\| {\Delta}_j \Bf\|_{L^2} \cdot {\sum_{k \geq j-1}}2^{sj}2^{j-k} \| {\Delta}_k \Bf\|_{L^2} \big)\mathrm{d}\tau\nonumber
  \\
  \lesssim & \int^t_0 \|\nabla \Bf\|_{L^\infty} \|  \Bf\|^2_{\dot{H}^s}\mathrm{d}\tau+ \int^t_0 \|\nabla \Bf\|_{L^\infty} {\sum_{j \geq -1}} \big( 2^{sj}\| {\Delta}_j \Bf\|_{L^2} \cdot {\sum_{k \geq -1}}2^{kj} \| {\Delta}_k \Bf\|_{L^2} \big)\mathrm{d}\tau\nonumber
  \\
  \lesssim & \int^t_0 \|\nabla \Bf\|_{L^\infty} \|  \Bf\|^2_{\dot{H}^s}\mathrm{d}\tau.
\end{align}
In a similar fashion, using \eqref{J1} and \eqref{J6}, we can show that
\begin{equation}\label{EJ1}
 \Big| \int^t_0 {\sum_{j \geq -1}}2^{2sj}  (J_1+J_7) \mathrm{d}\tau  \Big|
  \lesssim  \int^t_0 \|(\nabla \Bf, \nabla \Bh)\|_{L^\infty} \|  (\Bf, \Bh) \|^2_{\dot{H}^s}\mathrm{d}\tau,
\end{equation}
and
\begin{equation}\label{EJ6}
   \Big| \int^t_0 {\sum_{j \geq -1}}2^{2sj}  J_6 \mathrm{d}\tau  \Big| \\
  \lesssim \int^t_0 \|(\nabla \Bf, \nabla \Bh)\|_{L^\infty} \|  (\Bf, \Bh) \|^2_{\dot{H}^s}\mathrm{d}\tau.
\end{equation}
Using \eqref{J2} and the Sobolev embedding: $L^\infty \hookrightarrow \dot{B}^{0}_{\infty, \infty}$ (c.f. \cite{BC} Proposition 2.39), we deduce that
\begin{align}\label{EJ2}
\Big| \int^t_0 {\sum_{j \geq -1}}2^{2sj}  J_2 \mathrm{d}\tau  \Big|  \lesssim & \int^t_0 {\sum_{j \geq -1}} \big\{  \|\nabla \bU\|_{L^\infty} 2^{2sj}\|{\Delta}_j \Bf\|^2_{L^2}+2^{sj}\|{\Delta}_j \bU\|_{L^\infty} \|\nabla \Bf\|_{L^2} 2^{sj}\|{\Delta}_j \Bf\|_{L^2} \nonumber \\
& \qquad\qquad + 2^{2sj}\|\nabla \bU\|_{L^\infty}\|{\Delta}_j \Bf\|_{L^2}{\sum_{k \geq j-1}} 2^{j-k} \|{{\Delta}}_k \Bf\|_{L^2} \big\} \mathrm{d}\tau \nonumber \\
\lesssim & \int^t_0 \|\nabla \bU\|_{L^\infty} \|  \Bf \|^2_{\dot{H}^s}+\| \bU\|_{\dot{B}^{s+1}_{\infty, \infty}} \|\nabla \Bf\|_{L^2} \| \Bf\|_{\dot{H}^{s}} \mathrm{d}\tau \nonumber \\
\lesssim & \int^t_0 \|\nabla \bU\|_{L^\infty} \|  \Bf \|^2_{\dot{H}^s}+\| \bU\|_{\dot{W}^{s+1, \infty}} \|\nabla \Bf\|_{L^2} \| \Bf\|_{\dot{H}^{s}} \mathrm{d}\tau \nonumber \\
\lesssim & \int^t_0 \|\bU\|_{W^{s+1,\infty}} \|  \Bf \|^2_{{H}^s} \mathrm{d}\tau.
\end{align}
In deriving \eqref{EJ2}, we used the following argument based on H\"older's inequality:
\begin{align}\label{SE}
 & {\sum_{j\geq -1} 2^{sj}} \|{\Delta}_j \bU\|_{L^\infty}  2^{sj}\|{\Delta}_j \Bf\|_{L^2} \nonumber\\
 \leq & \sup_{j \geq -1 }\big(2^{(s+1)j}\|{\Delta}_j \bU\|_{L^\infty}\big)  \big( {\sum_{j\geq -1}  (2^{sj}}  \|{\Delta}_j \Bf\|_{L^2})^2 \big)^{\frac{1}{2}} \big( {\sum_{j\geq -1} }  2^{-2j} \big)^{\frac{1}{2}}
  \nonumber\\
  \leq &  \sup_{j \geq -1 } \big(2^{(s+1)j}\|{\Delta}_j \bU\|_{L^\infty}\big)  \big( {\sum_{j\geq -1} }  ( 2^{sj}\|{\Delta}_j \Bf\|_{L^2})^2 \big)^{\frac{1}{2}}
  \nonumber\\
  = & \| \bU\|_{\dot{B}^{s+1}_{\infty, \infty}} \| \Bf\|_{\dot{H}^{s}}.
\end{align}
In a similar fashion, using \eqref{J3} and \eqref{J5}, we can show that
\begin{align}\label{EJ3}
 & \Big| \int^t_0 {\sum_{j \geq -1}}2^{2sj}  (J_3+J_5+J_9+J_{11}) \mathrm{d}\tau  \Big|
 \nonumber \\
  \lesssim & \int^t_0 {\sum_{j \geq -1}} \big\{  2^{js}\| {\Delta}_j (\Bf \cdot \nabla \bU) \|_{L^2} \cdot  2^{js} \|{\Delta}_j \Bf\|_{L^2}+   2^{js}\| {\Delta}_j (\Bh \cdot \nabla \bB) \|_{L^2} \cdot  2^{js}\|{\Delta}_j \Bf\|_{L^2}
  \nonumber\\
  & \qquad\qquad +2^{js} \| {\Delta}_j (\Bf \cdot \nabla \bB) \|_{L^2}\cdot 2^{js} \|{\Delta}_j \Bh\|_{L^2}+  2^{js}\| {\Delta}_j (\Bh \cdot \nabla \bU) \|_{L^2} \cdot 2^{js} \|{\Delta}_j \Bh\|_{L^2} \big \} \mathrm{d}\tau
  \nonumber\\
  \lesssim & \int^t_0 \big( \| \Bf \cdot \nabla \bU \|_{\dot{H}^s} \| \Bf\|_{\dot{H}^s}+\| \Bh \cdot \nabla \bB \|_{\dot{H}^s} \| \Bf\|_{\dot{H}^s} \big) \mathrm{d}\tau
  \nonumber\\
  & +\int^t_0 \big( \| \Bf \cdot \nabla \bB \|_{\dot{H}^s} \| \Bh\|_{\dot{H}^s}+\| \Bh \cdot \nabla \bU \|_{\dot{H}^s} \| \Bh\|_{\dot{H}^s} \big) \mathrm{d}\tau
  \nonumber\\
  \lesssim & \int^t_0 \| (\bU,  \bB) \|_{{W}^{s+1,\infty}} \big(\| \Bf \|^2_{\dot{H}^s}+\| \Bh \|^2_{\dot{H}^s}\big)  \mathrm{d}\tau.
\end{align}
Using \eqref{J4} and \eqref{J8}, we can show that
\begin{align}\label{EJ4}
&\Big| \int^t_0 {\sum_{j \geq -1}}2^{2sj}  (J_4+J_{10}) \mathrm{d}\tau  \Big| \nonumber\\
\lesssim & \int^t_0  \big\{ \|\nabla \bB\|_{L^\infty} \big( \| \Bh\|^2_{\dot{H}^s}+\| \Bf\|^2_{\dot{H}^s}\big) +\| \bB\|_{\dot{B}^{s+1}_{\infty, \infty}} \big( \|\nabla \Bh\|_{L^2}+ \|\nabla \Bf\|_{L^2}\big) \|\Bf\|_{\dot{H}^s}\nonumber\\
&\qquad +\|\nabla \bB\|_{L^\infty}  \| \Bf\|_{\dot{H}^s} \| \Bh\|_{\dot{H}^s} \big\} \mathrm{d}\tau
\nonumber\\
 \lesssim & \int^t_0   \|\bB\|_{W^{s+1,\infty}} ( \| \Bh\|^2_{{H}^s}+\| \Bf\|^2_{{H}^s}) \mathrm{d}\tau,
\end{align}
and
\begin{align}\label{EJ8}
 & \Big| \int^t_0 {\sum_{j \geq -1}}2^{2sj}  J_8 \mathrm{d}\tau  \Big| \nonumber \\
 \lesssim  & \int^t_0  \big\{ \|\nabla \bB\|_{L^\infty} \big( \| \Bh\|^2_{\dot{H}^s}+\| \Bf\|^2_{\dot{H}^s}\big)+\| \bB\|_{\dot{B}^{s+1}_{\infty, \infty}}  \|\nabla \Bh\|_{L^2} \|\Bf\|_{\dot{H}^s}+\|\nabla \bB\|_{L^\infty}  \| \Bf\|_{\dot{H}^s} \| \Bh\|_{\dot{H}^s} \big\} \mathrm{d}\tau \nonumber \\
  \lesssim & \int^t_0   \|\bB\|_{W^{s+1,\infty}} ( \| \Bh\|^2_{{H}^s}+\| \Bf\|^2_{{H}^s}) \mathrm{d}\tau.
\end{align}
Lastly, using \eqref{J12}, we can show that
\begin{align}\label{EJ121}
  \Big| \int^t_0 {\sum_{j \geq -1}}2^{2sj}  J_{12} \mathrm{d}\tau  \Big| \lesssim   &\int^t_0 \big\{\big( \| (\nabla \times {\bB} ) \times {\bB} \|_{\dot{H}^{s+\frac{1}{2}}} + \| (\nabla \times {\bB} ) \times {\Bh} \|_{\dot{H}^{s+\frac{1}{2}}} \big) \| \Bh\|_{\dot{H}^{s+\frac{1}{2}}}\nonumber\\
  & \qquad\quad + \| \bB\|_{\dot{B}^{s+\frac{3}{2}}_{\infty,\infty}} \|\nabla \Bh\|_{L^{2}} \|\nabla \Bh\|_{\dot{H}^{s-\frac{1}{2}}}\big\}\mathrm{d}\tau\nonumber \\
   & + {\sum_{j \geq -1}} 2^{(2s+1)j} {\int^t_0} \|\nabla \Bh\|_{L^\infty}\big(  \| {\Delta}_j \Bh\|^2_{L^2}+\big( {\sum_{k \geq j-1}} 2^{j-k} \| {\Delta}_k \Bh\|_{L^2}  \big)^2 \big) \mathrm{d}\tau\nonumber\\
  \lesssim   &\int^t_0\big\{ \big( \| (\nabla \times {\bB} ) \times {\bB} \|_{\dot{H}^{s+\frac{1}{2}}} + \| \nabla {\bB} \|_{W^{s+\frac{1}{2}, \infty}} \|{\Bh} \|_{\dot{H}^{s+\frac{1}{2}}} \big) \| \Bh\|_{\dot{H}^{s+\frac{1}{2}}}\nonumber\\
  & \qquad\quad + \| \bB\|_{W^{s+\frac{3}{2},\infty}} \|\nabla \Bh\|_{L^{2}} \|\nabla \Bh\|_{\dot{H}^{s-\frac{1}{2}}}\big\}\mathrm{d}\tau\nonumber\\
   & + {\sum_{j \geq -1}} 2^{(2s+1)j} {\int^t_0} \|\nabla \Bh\|_{L^\infty}\big(  \| {\Delta}_j \Bh\|^2_{L^2}+\big( {\sum_{k \geq j-1}} 2^{j-k} \| {\Delta}_k \Bh\|_{L^2}  \big)^2 \big) \mathrm{d}\tau\nonumber\\
  \lesssim   &\int^t_0 \big\{\big( \| (\nabla \times {\bB} ) \times {\bB} \|_{\dot{H}^{s+\frac{1}{2}}} + \|  {\bB} \|_{W^{s+\frac{3}{2}, \infty}} \| \Lambda^{\frac{1}{2}}{\Bh} \|_{\dot{H}^{s}} \big) \| \Lambda^{\frac{1}{2}}{\Bh} \|_{\dot{H}^{s}} \nonumber\\
  &\qquad\quad  +  \|\bB\|_{W^{s+\frac{3}{2},\infty}} \|\nabla \Bh\|_{L^{2}} \| \Lambda^{\frac{1}{2}}{\Bh} \|_{\dot{H}^{s}}\big\} \mathrm{d}\tau\nonumber\\
   & + {\sum_{j \geq -1}} 2^{(2s+1)j} {\int^t_0} \|\nabla \Bh\|_{L^\infty}\big(  \| {\Delta}_j \Bh\|^2_{L^2}+\big( {\sum_{k \geq j-1}} 2^{j-k} \| {\Delta}_k \Bh\|_{L^2}  \big)^2\big) \mathrm{d}\tau,
\end{align}
where we used the qualitative equivalence:
$$
\| \Lambda^{\frac{1}{2}}{\Bh} \|_{\dot{H}^{s}} \simeq  \|{\Bh} \|_{\dot{H}^{s+\frac{1}{2}}}\simeq  \|\nabla{\Bh} \|_{\dot{H}^{s-\frac{1}{2}}}.
$$
Using the Gagliardo-Nirenberg inequality:
$$
\| \nabla \Bh \|_{L^2} \lesssim \|\Lambda^{\frac{1}{2}} \Bh\|_{L^2}^{1-\frac{1}{2s}} \| \Lambda^{\frac{1}{2}+s} \Bh\|^{\frac{1}{2s}}_{L^2}\lesssim \| \Lambda^{\frac{1}{2}} \Bh\|_{H^s}, \quad s>\frac52,
$$
we update \eqref{EJ121} as
\begin{align}\label{EJ12}
  \Big| \int^t_0 {\sum_{j \geq -1}}2^{2sj}  J_{12} \mathrm{d}\tau  \Big|
   \lesssim   &\int^t_0 \big( \| (\nabla \times {\bB} ) \times {\bB} \|_{\dot{H}^{s+\frac{1}{2}}} + \|  {\bB} \|_{W^{s+\frac{3}{2}, \infty}} \| \Lambda^{\frac{1}{2}}{\Bh} \|_{{H}^{s}} \big) \| \Lambda^{\frac{1}{2}}{\Bh} \|_{\dot{H}^{s}} \mathrm{d}\tau \nonumber \\
   & + {\sum_{j \geq -1}} 2^{(2s+1)j} {\int^t_0} \|\nabla \Bh\|_{L^\infty}\big(  \| {\Delta}_j \Bh\|^2_{L^2}+\big( {\sum_{k \geq j-1}} 2^{j-k} \| {\Delta}_k \Bh\|_{L^2}  \big)^2\big) \mathrm{d}\tau.
\end{align}
Assembling the estimates \eqref{EJ0}, \eqref{EJ1}, \eqref{EJ6},\eqref{EJ2}, \eqref{EJ3}, \eqref{EJ4}, \eqref{EJ8}, and \eqref{EJ12}, we obtain
\begin{align}\label{highe4}
\Big | \int^t_0 {\sum_{j \geq -1}}2^{2sj} {\sum_{k=0}^{12}} J_k \mathrm{d}\tau  \Big|  &\lesssim \int^t_0 \|(\nabla \Bf,\nabla\Bh)\|_{L^\infty} \big(\|  \Bf\|^2_{\dot{H}^s}+ \|  \Bh \|^2_{\dot{H}^s} \big) \mathrm{d}\tau \nonumber\\
  & +  {\sum_{k \geq j-1}} 2^{(2s+1)j} {\int^t_0} \|\nabla \Bh\|_{L^\infty}\big(  \| {\Delta}_j \Bh\|^2_{L^2}+\big( {\sum_{k \geq j-1}} 2^{j-k} \| {\Delta}_k \Bh\|_{L^2} \big )^2 \big) \mathrm{d}\tau\nonumber\\
  & +  \int^t_0 \big( \| (\nabla \times \bB) \times \bB \|_{\dot{H}^{s+\frac{1}{2}}} \|\Lambda^{\frac{1}{2}} \Bh\|_{\dot{H}^s} +   \|  \bB \|_{W^{s+\frac{3}{2},\infty}} \|\Lambda^{\frac{1}{2}} \Bh\|_{{H}^s} \|\Lambda^{\frac{1}{2}} \Bh\|_{\dot{H}^s} \big) \mathrm{d}\tau\nonumber\\
  &+  \int^t_0 \|(\bU, \bB)\|_{W^{s+1,\infty}} \|(\Bh,\Bf)\|^2_{{H}^s} \mathrm{d}\tau.
\end{align}
Substituting \eqref{highe3} and \eqref{highe4} to \eqref{highe2}, we get
\begin{align}\label{h2}
 & \| \Bf(t) \|^2_{\dot{H}^s}+ \| \Bh(t) \|^2_{\dot{H}^s}+ \int^t_0 \| \Lambda^{\alpha} \Bf(\tau) \|^2_{\dot{H}^s} d \tau+  \int^t_0 \| \Lambda^{\frac{1}{2}} \Bh(\tau) \|^2_{\dot{H}^s} \mathrm{d} \tau \nonumber\\
  \lesssim \ & \|  \bu_{01} \|^2_{\dot{H}^s}+ \|  \bb_{01} \|^2_{\dot{H}^s}+ \int^t_0 \|(\nabla \Bf,\nabla\Bh)\|_{L^\infty} \big(\| \Bf\|^2_{\dot{H}^s}+ \| \Bh \|^2_{\dot{H}^s} \big) \mathrm{d}\tau\nonumber\\
  & +  {\sum_{j \geq -1}} 2^{(2s+1)j} \displaystyle{\int^t_0} \|\nabla \Bh\|_{L^\infty}\big(  \| {\Delta}_j \Bh\|^2_{L^2}+( {\sum_{k \geq j-1}} 2^{j-k} \| {\Delta}_k \Bh\|_{L^2}  )^2 \big) \mathrm{d}\tau\nonumber\\
  & +  \int^t_0 \big( \| (\nabla \times \bB) \times \bB \|_{\dot{H}^{s+\frac{1}{2}}} \|\Lambda^{\frac{1}{2}} \Bh\|_{\dot{H}^s} +   \|  \bB \|_{W^{s+\frac{3}{2},\infty}} \|\Lambda^{\frac{1}{2}} \Bh\|_{H^s}\|\Lambda^{\frac{1}{2}} \Bh\|_{\dot{H}^s} \big) \mathrm{d}\tau\nonumber\\
  &+  \int^t_0 \|(\bU,  \bB)\|_{W^{s+1,\infty}} \|(\Bh,\Bf)\|^2_{{H}^s} \mathrm{d}\tau +\int^t_0 \big( \|\bF\|_{\dot{H}^s} \|\Bf\|_{\dot{H}^s}+ \|\bG\|_{\dot{H}^s} \|\Bh\|_{\dot{H}^s} \big) \mathrm{d}\tau.
\end{align}
By using the Gagliardo-Nirenberg inequalities (c.f. \eqref{GN}) and Sobolev embedding
$$
\|(\nabla \Bf,\nabla\Bh)\|_{L^\infty} \leq \| (\Bf,\Bh)\|_{H^s}, \quad s>\frac{5}{2},
$$
we can show that
\begin{align}\label{GN0}
  \|(\nabla \Bf,\nabla\Bh)\|_{L^\infty} \big(\| \Bf\|^2_{\dot{H}^s}+ \| \Bh \|^2_{\dot{H}^s} \big)
  &= \|(\nabla \Bf,\nabla\Bh)\|_{L^\infty} \big(\| \Lambda^s \Bf \|^2_{L^2} + \| \Lambda^{s} \Bh \|^2_{L^2}\big)\nonumber\\
  &\lesssim  \| (\Bf,\Bh)\|_{H^s}\big(\| \Lambda^\alpha \Bf \|^2_{H^s} + \| \Lambda^{\frac{1}{2}} \Bh \|^2_{H^s}\big).
\end{align}
Inserting \eqref{GN0} into \eqref{h2}, we update \eqref{GN0} as
\begin{align}\label{h21}
 & \| \Bf(t) \|^2_{\dot{H}^s}+ \| \Bh(t) \|^2_{\dot{H}^s}+ \int^t_0 \| \Lambda^{\alpha} \Bf(\tau) \|^2_{\dot{H}^s} d \tau+  \int^t_0 \| \Lambda^{\frac{1}{2}} \Bh(\tau) \|^2_{\dot{H}^s} \mathrm{d} \tau \nonumber\\
  \lesssim \ & \|  \bu_{01} \|^2_{\dot{H}^s}+ \|  \bb_{01} \|^2_{\dot{H}^s}+ \int^t_0 \| (\Bf,\Bh)\|_{H^s}\big(\| \Lambda^\alpha \Bf \|^2_{H^s} + \| \Lambda^{\frac{1}{2}} \Bh \|^2_{H^s}\big) \mathrm{d}\tau\nonumber\\
  & +   \displaystyle{\int^t_0} \|\Bh\|_{H^s}\big(  \| \Lambda^{\frac{1}{2}} \Bh\|^2_{\dot{H}^s}+{\sum_{j \geq -1}} 2^{(2s+1)j} ( {\sum_{k \geq j-1}} 2^{j-k} \| {\Delta}_k \Bh\|_{L^2}  )^2 \big) \mathrm{d}\tau\nonumber\\
  & +  \int^t_0 \big( \| (\nabla \times \bB) \times \bB \|_{\dot{H}^{s+\frac{1}{2}}} \|\Lambda^{\frac{1}{2}} \Bh\|_{\dot{H}^s} +   \|  \bB \|_{W^{s+\frac{3}{2},\infty}} \|\Lambda^{\frac{1}{2}} \Bh\|_{{H}^s} \|\Lambda^{\frac{1}{2}} \Bh\|_{\dot{H}^s} \big) \mathrm{d}\tau\nonumber\\
  &+  \int^t_0 \|(\bU,  \bB)\|_{W^{s+1,\infty}} \|(\Bh,\Bf)\|^2_{{H}^s} \mathrm{d}\tau +\int^t_0 \big( \|\bF\|_{\dot{H}^s} \|\Bf\|_{\dot{H}^s}+ \|\bG\|_{\dot{H}^s} \|\Bh\|_{\dot{H}^s} \big) \mathrm{d}\tau.
\end{align}
For the term involving summation on the right-hand side of \eqref{h21}, using Young's inequality for series convolution, we deduce that
\begin{align}\label{yc}
 {\sum_{j \geq -1}} 2^{(2s+1)j}  \big( {\sum_{k \geq j-1}} 2^{j-k} \| {\Delta}_k \Bh\|_{L^2}  \big)^2 = & {\sum_{j \geq -1}}  \big( {\sum_{k \geq j-1}} 2^{(j-k)(s+\frac{3}{2})} 2^{(s+\frac{1}{2})k}\| {\Delta}_k \Bh\|_{L^2}  \big)^2 \nonumber\\
 \lesssim & {\sum_{j \geq -1}} 2^{(2s+1)j} \| {\Delta}_j \Bh\|^2_{L^2}  = \|\Lambda^{\frac{1}{2}}\Bh\|^2_{\dot{H}^s}.
\end{align}
Using \eqref{yc}, we update \eqref{h21} as
\begin{align}\label{h22}
 & \| \Bf(t) \|^2_{\dot{H}^s}+ \| \Bh(t) \|^2_{\dot{H}^s}+ \int^t_0 \| \Lambda^{\alpha} \Bf(\tau) \|^2_{\dot{H}^s} d \tau+  \int^t_0 \| \Lambda^{\frac{1}{2}} \Bh(\tau) \|^2_{\dot{H}^s} \mathrm{d} \tau \nonumber\\
  \lesssim \ & \|  \bu_{01} \|^2_{\dot{H}^s}+ \|  \bb_{01} \|^2_{\dot{H}^s}+ \int^t_0 \| (\Bf,\Bh)\|_{H^s}\big(\| \Lambda^\alpha \Bf \|^2_{H^s} + \| \Lambda^{\frac{1}{2}} \Bh \|^2_{H^s}\big) \mathrm{d}\tau\nonumber\\
  & +   \displaystyle{\int^t_0} \|\Bh\|_{H^s} \| \Lambda^{\frac{1}{2}} \Bh\|^2_{\dot{H}^s}\mathrm{d}\tau+  \int^t_0 \|(\bU,  \bB)\|_{W^{s+1,\infty}} \|(\Bh,\Bf)\|^2_{\dot{H}^s} \mathrm{d}\tau\nonumber\\
  & +  \int^t_0 \big( \| (\nabla \times \bB) \times \bB \|_{\dot{H}^{s+\frac{1}{2}}} \|\Lambda^{\frac{1}{2}} \Bh\|_{\dot{H}^s} +   \|  \bB \|_{W^{s+\frac{3}{2},\infty}} \|\Lambda^{\frac{1}{2}} \Bh\|_{{H}^s} \|\Lambda^{\frac{1}{2}} \Bh\|_{\dot{H}^s} \big) \mathrm{d}\tau\nonumber\\
  & +\int^t_0 \big( \|\bF\|_{\dot{H}^s} \|\Bf\|_{\dot{H}^s}+ \|\bG\|_{\dot{H}^s} \|\Bh\|_{\dot{H}^s} \big) \mathrm{d}\tau.
\end{align}
Adding \eqref{h22} to \eqref{L2}, we can show that
\begin{align}\label{fi}
  &\| \Bf(t) \|^2_{H^s}+ \| \Bh(t) \|^2_{H^s}+ \int^t_0 \| \Lambda^{\alpha} \Bf(\tau) \|^2_{H^s} \mathrm{d} \tau+  \int^t_0 \| \Lambda^{\frac{1}{2}} \Bh(\tau) \|^2_{H^s}  \mathrm{d}\tau \nonumber\\
     \lesssim & \|  \bu_{01} \|^2_{H^s}+ \| \bb_{01} \|^2_{H^s}+  \sup_{\tau \in [0,t]}\| (\Bf,\Bh)(\tau)\|_{H^s} \int^t_0 \big(\| \Lambda^\alpha \Bf \|^2_{H^s} + \| \Lambda^{\frac{1}{2}} \Bh \|^2_{H^s}\big) \mathrm{d}\tau\nonumber\\
   & +  \big(\int^t_0  \| (\nabla \times \bB) \times \bB \|^2_{H^{s+\frac{1}{2}}} \mathrm{d}\tau\big)^{\frac{1}{2}} \big(\int^t_0 \|\Lambda^{\frac{1}{2}} \Bh\|^2_{H^s} \mathrm{d}\tau \big)^{\frac{1}{2}}
   \nonumber\\
   & +  \big(\sup_{t\geq 0}\|  \bB (t) \|_{W^{s+\frac{3}{2},\infty}}\big) \int^t_0    \| \Lambda^{\frac{1}{2}} \Bh\|^2_{H^s}  \mathrm{d}\tau
   +  \int^t_0 \|(\bU, \bB)\|_{W^{s+1,\infty}} \|(\Bh,\Bf)\|^2_{H^s}\mathrm{d}\tau
 \nonumber \\
  & + \sup_{\tau \in [0,t]} \big(\|\Bf(\tau)\|_{H^s}+\|\Bh(\tau)\|_{H^s}\big) \int^t_0 \left( \|\bF\|_{H^s} + \|\bG\|_{H^s} \right) \mathrm{d}\tau.
\end{align}
Next, we derive a global energy bound for $(\Bf,\Bh)$, by using the estimate of $(\bU,\bB)$, {\it a priori} smallness assumption, and standard continuation argument.

{\bf Step 6.} Set
\begin{equation*}
  E(t)=\| \Bf(t) \|^2_{{H}^s}+ \| \Bh(t) \|^2_{{H}^s}+ \int^t_0 \|\Lambda^{\alpha} \Bf(\tau) \|^2_{{H}^s} \mathrm{d} \tau +\int^t_0 \| \Lambda^{\frac{1}{2}} \Bh(\tau) \|^2_{{H}^s} \mathrm{d} \tau.
\end{equation*}
Then it follows from \eqref{fi} that
\begin{align}\label{EW}
  E(t) \lesssim  & \| {\bu}_{01} \|^2_{H^s}+ \| {\bb}_{01} \|^2_{H^s} + \underbrace{\big(\sup_{\tau\in [0,t]} E(\tau)\big)^{\frac{1}{2}} \, E(t)}_{\equiv R_1} + \underbrace{\| (\nabla \times \bB) \times \bB \|_{L^2_t H^{s+\frac{1}{2}}} E^{\frac{1}{2}}(t)}_{\equiv R_2}\nonumber
  \\
  & + \underbrace{\int^t_0  \|(\bU, \bB)\|_{ W^{s+1,\infty}}\, E(\tau) \mathrm{d}\tau}_{\equiv R_3}+\underbrace{\big(\sup_{t\in[0,T]}\|  \bB (t) \|_{W^{s+\frac{3}{2},\infty}}\big)\, E(t)}_{\equiv R_4}\nonumber
  \\
  & + \underbrace{\big (\sup_{\tau\in [0,t]} E(\tau)\big)^{\frac{1}{2}} \int^t_0  \big (\|\bF\|_{H^s} + \|\bG\|_{H^s} \big)   \mathrm{d}\tau}_{\equiv R_5},
\end{align}
where $t\in(0,T]$, and $T>0$ is within the lifespan of local solution. Suppose $\sup_{t\in[0,T]}E(t)$ is sufficiently small.

First, note that when $\sup_{t\in[0,T]}E(t)$ is sufficiently small, $R_1$ can be absorbed by the left-hand side of \eqref{EW}. Second, using the fact $\mathrm{supp}\,\widehat{\bB} (t,\cdot) \subseteq \{\bm \xi: 1-\varepsilon < |\bm \xi| < 1+\varepsilon \}$, $0< \varepsilon< \frac{1}{2}$, and \eqref{111} and \eqref{UB1}, we can show that
\begin{equation*}
  \|  \bB (t) \|_{W^{s+\frac{3}{2},\infty}} \lesssim \|  \widehat{\bB} (t,\cdot) \|_{L^1_{\bm \xi}} \lesssim  \big\| e^{-|\bm\xi|t} \gamma(\bm \xi)\widehat{\bv}_0 \big\|_{L^1_{\bm \xi}} \lesssim \|\widehat{\bv}_0 \|_{L^1_{\bm \xi}} \lesssim \delta.
\end{equation*}
Hence, when $\sup_{t\in[0,T]}E(t)$ is sufficiently small, $R_4$ can be absorbed by the left-hand side of \eqref{EW}. Using Cauchy's inequality, we can update \eqref{EW} as
\begin{align}\label{Z1}
 E(t) \lesssim  & \|  {\bu}_{01} \|^2_{H^s}+ \| {\bb}_{01} \|^2_{H^s}
   + \underbrace{\| (\nabla \times \bB) \times \bB \|_{L^2_t H^{s+\frac{1}{2}}}^2}_{\equiv R_2'} \nonumber\\
  & + \underbrace{\int^t_0  \|(\bU, \bB)\|_{ W^{s+1,\infty}} \, E(\tau) \mathrm{d}\tau}_{R_3} + \underbrace{\big(\sup_{\tau\in [0,t]} E(\tau)\big)^{\frac{1}{2}} \int^t_0 \big ( \|\bF\|_{H^s} + \|\bG\|_{H^s} \big) \mathrm{d}\tau }_{R_5}.
\end{align}
Note that according to \eqref{U1} and \eqref{111}, we have
\begin{align}\label{Z2}
R_2' \lesssim \big(\varepsilon^{s+\frac{1}{2}} \|\widehat{\bv}_0\|_{L^2_{\bm\xi}}\|\widehat{\bv}_0\|_{L^1_{\bm\xi}}\big)^2 \lesssim \delta \big(\varepsilon^{s+\frac{1}{2}} \|\widehat{\bv}_0\|_{L^2_{\bm\xi}}\|\widehat{\bv}_0\|_{L^1_{\bm\xi}}\big).
\end{align}
Recall that $\bU$ and $\bB$ satisfy \eqref{U} and \eqref{B}, respectively. Using the information about the support of $\widehat{\bU}$ and $\widehat{\bB}$, we can show that
\begin{equation}\label{Z3}
\begin{split}
   \|(\bU, \bB)\|_{ W^{s+1,\infty}}  & \lesssim  (1+|\bm\xi |^2)^{\frac{s}{2}} \|(\widehat \bU, \widehat\bB) \|_{ L^1_{\bm\xi}}  \lesssim  \|(\widehat \bU, \widehat\bB) \|_{ L^1_{\bm\xi}}.
\end{split}
\end{equation}
Moreover, recall that $\bF= {\bU}\times (\nabla \times {\bU})-{\bB}\times \left(\nabla \times {\bB}\right), \quad \bG= \nabla \times ( {\bU} \times {\bB})$. Using similar arguments as those in the proof of Lemma \ref{L1}, we can show that
\begin{align}\label{Z4}
\int^t_0 \big( \|\bF\|_{H^s} + \|\bG\|_{H^s} \big) \mathrm{d}\tau \lesssim (\varepsilon^s +\varepsilon )\|\widehat{\bv}_0\|_{L^2_{\bm\xi}}\|\widehat{\bv}_0\|_{L^1_{\bm\xi}} \lesssim \varepsilon \|\widehat{\bv}_0\|_{L^2_{\bm\xi}}\|\widehat{\bv}_0\|_{L^1_{\bm\xi}},
\end{align}
where we used the fact that $0<\varepsilon < 1$.  Using \eqref{Z2}--\eqref{Z4}, we update \eqref{Z1} as
\begin{equation*}
  \begin{split}
  E(t) \lesssim & \|  {\bu}_{01} \|^2_{H^s}+ \| {\bb}_{01} \|^2_{H^s} + \delta \big(\varepsilon^{s+\frac{1}{2}} \|\widehat{\bv}_0\|_{L^2_{\bm\xi}}\|\widehat{\bv}_0\|_{L^1_{\bm\xi}}\big) + \big(\sup_{\tau\in [0,t]} E(\tau)\big)^{\frac{1}{2}}\, \varepsilon\, \|\widehat{\bv}_0\|_{L^2_{\bm\xi}}\|\widehat{\bv}_0\|_{L^1_{\bm\xi}}\\
  & + \int^t_0 \| (\widehat \bU, \widehat{\bB}) \|_{L^1_{\bm\xi}} E(\tau) \mathrm{d}\tau.
  \end{split}
\end{equation*}
Note that according to \eqref{UB1}, it holds that
\begin{align}\label{Z5}
\int^t_0 \| (\widehat \bU, \widehat{\bB}) \|_{L^1_{\bm\xi}} \mathrm{d}\tau \lesssim \|\widehat{\bv}_0\|_{L^1_{\bm\xi}}.
\end{align}
Therefore, the global energy bound of $E(t)$ follows from Gr\"onwall's inequality, the smallness of $\delta$, {\it a priori} smallness assumption on $\sup_{t\in[0,T]}E(t)$, \eqref{Z5}, \eqref{111}, and standard continuation argument. Combing the global energy bound of $E(t)$ and the local existence of its solutions (please see Remark \ref{rf} and Section 3.2), we can complete the proof of Theorem \ref{thm2}. \hfill$\square$
\begin{remark}\label{rf}
The local-in-time existence of Theorem \ref{thm2} can be obtained similarly as in Section 3.2, for we can use the similar approximation procedure, Picard's theorem and the global bound of $E(t)$ to get a limit $(\Bf, \Bh)$ for \eqref{fh}. Then, by using interpolation inequality and a strong convergence to conclude that $(\Bf, \Bh) \in C([0,1]; H^a), 0<a<s$. Then $(\Bf, \Bh)$ satisfies the equation \eqref{fh}, and it's the local solution of \eqref{fh}.
\end{remark}

\section*{Acknowledgements}
The author would like to express thanks for the reviewers for much helpful advice. The research of H.-L. Zhang was partially supported by Hunan Provincial Key Laboratory of Intelligent Processing of Big Data on Transportation, Changsha University of Science and Technology. K. Zhao was partially supported by the Simons Foundation Collaboration Grant for Mathematicians (No. 413028).
\section*{Conflicts of interest}
The authors declared that this work does not have any conflicts of interest.

\end{document}